\documentclass[11pt]{amsart}

\usepackage{enumerate}
\usepackage{amsmath,amsthm,verbatim,amssymb,amsfonts,amscd,amsopn,amsxtra,graphicx,lmodern}
\usepackage{hyperref}
\usepackage{tikz-cd} 
\usepackage{graphics}
\usepackage{mathrsfs}
\usepackage{relsize}
\usepackage{enumitem}
\usepackage[utf8]{inputenc}
\usepackage{csquotes}
\usepackage{amsthm}
\usepackage{thmtools}
\usepackage{mathtools}
\usepackage{microtype}
\usepackage{pdfrender,xcolor}
\usepackage[sort cites=true, backend=biber]{biblatex}
\usepackage{lua-visual-debug}
\usepackage{biblatex}

\tikzset{
	symbol/.style={
		draw=none,
		every to/.append style={
			edge node={node [sloped, allow upside down, auto=false]{$#1$}}}
	}
}

\begin{document}
	\pdfrender{StrokeColor=black,TextRenderingMode=2,LineWidth=0.2pt}	
	
	\title{A ruled residue theorem for algebraic function fields of curves of prime degree}

	\author{Arpan Dutta}

	\def\NZQ{\mathbb}               
	\def\NN{{\NZQ N}}
	\def\QQ{{\NZQ Q}}
	\def\ZZ{{\NZQ Z}}
	\def\RR{{\NZQ R}}
	\def\CC{{\NZQ C}}
	\def\AA{{\NZQ A}}
	\def\BB{{\NZQ B}}
	\def\PP{{\NZQ P}}
	\def\FF{{\NZQ F}}
	\def\GG{{\NZQ G}}
	\def\HH{{\NZQ H}}
	\def\UU{{\NZQ U}}
	\def\P{\mathcal P}
	
	%
	%
	\let\union=\cup
	\let\sect=\cap
	\let\dirsum=\oplus
	\let\tensor=\otimes
	\let\iso=\cong
	\let\Union=\bigcup
	\let\Sect=\bigcap
	\let\Dirsum=\bigoplus
	\let\Tensor=\bigotimes
	
	\theoremstyle{plain}
	\newtheorem{Theorem}{Theorem}[section]
	\newtheorem{Lemma}[Theorem]{Lemma}
	\newtheorem{Corollary}[Theorem]{Corollary}
	\newtheorem{Proposition}[Theorem]{Proposition}
	\newtheorem{Problem}[Theorem]{}
	\newtheorem{Conjecture}[Theorem]{Conjecture}
	\newtheorem{Question}[Theorem]{Question}
	
	\theoremstyle{definition}
	\newtheorem{Example}[Theorem]{Example}
	\newtheorem{Examples}[Theorem]{Examples}
	\newtheorem{Definition}[Theorem]{Definition}
	
	\theoremstyle{remark}
	\newtheorem{Remark}[Theorem]{Remark}
	\newtheorem{Remarks}[Theorem]{Remarks}
	
	\newcommand{\trdeg}{\mbox{\rm trdeg}\,}
	\newcommand{\rr}{\mbox{\rm rat rk}\,}
	\newcommand{\sep}{\mathrm{sep}}
	\newcommand{\ac}{\mathrm{ac}}
	\newcommand{\ins}{\mathrm{ins}}
	\newcommand{\res}{\mathrm{res}}
	\newcommand{\Gal}{\mathrm{Gal}\,}
	\newcommand{\ch}{\operatorname{char}}
	\newcommand{\Aut}{\mathrm{Aut}\,}
	\newcommand{\kras}{\mathrm{kras}\,}
	\newcommand{\dist}{\mathrm{dist}\,}
	\newcommand{\ord}{\mathrm{ord}\,}

	\newcommand{\n}{\par\noindent}
	\newcommand{\nn}{\par\vskip2pt\noindent}
	\newcommand{\sn}{\par\smallskip\noindent}
	\newcommand{\mn}{\par\medskip\noindent}
	\newcommand{\bn}{\par\bigskip\noindent}
	\newcommand{\pars}{\par\smallskip}
	\newcommand{\parm}{\par\medskip}
	\newcommand{\parb}{\par\bigskip}
	\let\phi=\varphi
	\let\kappa=\varkappa
	
	\def \a {\alpha}
	\def \b {\beta}
	\def \s {\sigma}
	\def \d {\delta}
	\def \g {\gamma}
	\def \o {\omega}
	\def \l {\lambda}
	\def \th {\theta}
	\def \D {\Delta}
	\def \G {\Gamma}
	\def \O {\Omega}
	\def \L {\Lambda}
	%
	%
	\textwidth=15cm \textheight=22cm \topmargin=0.5cm
	\oddsidemargin=0.5cm \evensidemargin=0.5cm \pagestyle{plain}


	\address{Discipline of Mathematics, School of Basic Sciences, IIT Bhubaneswar, Argul,
		Odisha, India, 752050.}
	\email{arpandutta@iitbbs.ac.in}
	
	\date{\today}
	
	\thanks{}
	
	
	
	\maketitle


\begin{abstract}
	The Ruled Residue Theorem asserts that given a ruled extension $(K|k,v)$ of valued fields, the residue field extension is also ruled. In this paper we analyse the failure of this theorem when we set $K$ to be algebraic function fields of certain curves of prime degree $p$, provided $p$ is coprime to the residue characteristic and $k$ contains a primitive $p$-th root of unity. Specifically, we consider function fields of the form $K= k(X)(\sqrt[p]{aX^p+bX+c})$ where $a\neq 0$. We provide necessary conditions for the residue field extension to be non-ruled which are formulated only in terms of the values of the coefficients. This provides a far-reaching generalization of a certain important result regarding non-ruled extensions for function fields of smooth projective conics.   
\end{abstract}


\section{Introduction}

In this present paper we work with Krull valuations and their extensions to algebraic function fields. We refer the reader to [\ref{ZS2}] for relevant background on valuation theory. An extension of valued fields $(K|k,v)$ is a field extension $K|k$, $v$ is a valuation on $K$ and $k$ is equipped with the restricted valuation. The value group of a valued field $(k,v)$ is denoted by $vk$ and the residue field by $kv$. The value of an element $a$ is denoted by $va$ and its residue by $av$. Given two subfields $F_1$ and $F_2$ of an overfield $\O$, we denote their compositum by $F_1.F_2$.

\pars Before we can formulate the primary motivation behind this article, we need to make a few definitions. An \emph{algebraic function field $K|k$ of dimension $n$} is a finitely generated field extension $K|k$ with $\trdeg[K:k]=n$. $K|k$ is said to be \textbf{ruled} if $K$ is purely transcendental over a \emph{finite} extension of $k$. Assume that $K|k$ has dimension one and let $v$ be a valuation on $K$. We say that the valued field extension $(K|k,v)$ is \emph{residue transcendental} if the residue field extension $Kv|kv$ satisfies $\trdeg[Kv:kv] = \trdeg[K:k]=1$. The Ruled Residue Theorem, first proved by Ohm [\ref{Ohm}], asserts that if $K|k$ is ruled, of dimension one and $(K|k,v)$ is residue transcendental, then the residue field extension $Kv|kv$ is also ruled. 

\pars Our goal in this article is to investigate the failure of the Ruled Residue Theorem for function fields of curves of prime degrees. Specifically, we consider function fields of the form $K:= k(X)(\sqrt[p]{aX^p + bX +c})$, where $X$ is transcendental over $k$, $p$ is a prime such that $p\neq\ch k$, $k$ contains a primitive $p$-th root of unity and $aX^p+bX+c \in k[X]$ is a polynomial of degree $p$. An initial result in this direction was obtained in [\ref{Gupta Becher ruled residue conics}] where the authors prove the following:

\begin{Theorem}
	Assume that $K|k$ is the function field of a smooth projective conic and let $v$ be a valuation on $k$ with $\ch kv\neq 2$. Then $v$ has at most one residue transcendental extension to $K$ such that $Kv|kv$ is not ruled.    
\end{Theorem}
Recall that function fields of smooth projective conics are of the form $K = k(X)(\sqrt{aX^2 +c})$. To the best of our understanding, the techniques employed by the authors in [\ref{Gupta Becher ruled residue conics}] rely on identifying non-ruled valuations as Gau{\ss} valuations of certain rational function subfields, with a more restrictive notion of Gau{\ss} valuation. However, it is unclear to us whether their methods generalize easily to the setup of this paper. Our analysis shows that the presence of the degree one monomial $bX$ in our case complicates matters even further. 

\pars A primary object in our study of the problem is the theory of minimal pairs. Let $(k(X)|k,v)$ be a residue transcendental extension of valued fields. Fix an extension of $v$ to $\overline{k}(X)$ where $\overline{k}$ is an algebraic closure of $k$. Then the set $v(X-\overline{k}) := \{ v(X-d)\mid d\in\overline{k} \}$ admits a maximum [\ref{APZ characterization of residual trans extns}, Proposition 1.1(c)]. Set $\g:= \max v(X-\overline{k})$. Take $\a\in\overline{k}$ such that
\sn (i) $v(X-\a)=\g$,
\n (ii) $[k(\a):k] \leq [k(\b):k]$ whenever $v(X-\b)=\g$.\\
Such a pair $(\a,\g)$ is said to be a \textbf{minimal pair of definition} for $(k(X)|k,v)$. It follows from [\ref{APZ characterization of residual trans extns}, Theorem 2.1] that $\g\in v\overline{k}$. Thus $\g$ is contained in the divisible hull of $vk$ by [\ref{Engler book}, Theorem 3.2.11]. The importance of minimal pairs is illustrated by the fact that the extension $(k(X)|k,v)$ is completely described by the datum $(\a,\g)$ (see [\ref{AP sur une classe}], [\ref{APZ characterization of residual trans extns}], [\ref{Dutta min fields implicit const fields}]). In particular, assume that $\a\in k$. Any polynomial $f(X)\in k[X]$ has a unique expression $f(X) = \sum_{i=0}^{n} c_i (X-\a)^i$. Then $vf = \min\{ vc_i + i\g \}$. If $\a$ can be chosen to be $0$ then we say that $v$ is a \textbf{Gau{\ss} valuation} and $(k(X)|k,v)$ is a Gau{\ss} extension. The fact that $\gamma = \max v(X-\overline{k})$ implies that $vX\leq \gamma$. If $vX = \gamma = v(X-\alpha)$, then $(0,\gamma)$ is also a minimal pair of definition. Thus $vX < \gamma$ whenever $v$ is not a Gauss valuation.

\pars We can now state the central theorem of this paper. By abuse of notation, we say that $v$ is a Gau{\ss} valuation if the subextension $(k(X)|k,v)$ of $(K|k,v)$ is a Gau{\ss} extension.

\begin{Theorem}\label{Thm non ruled necessary}
	Let $K = k(X)(\sqrt[p]{aX^p+bX+c})$ where $p$ is a prime number such that $p\neq \ch k$, $a,b,c \in k$ and $a\neq 0$. Assume that $k$ contains a primitive $p$-th root of unity. Take a valuation $v$ of $k$ such that $p\neq \ch kv$ and a residue transcendental extension of $v$ to $K$. Assume that $Kv|kv$ is not ruled. Then 
	\[ vX = \min  \mathlarger{\mathlarger{\mathlarger{\{}}} \dfrac{1}{p-1} v(\dfrac{b}{a}) , \dfrac{1}{p} v(\dfrac{c}{a}) \mathlarger{\mathlarger{\mathlarger{\}}}}. \]
	If $\dfrac{1}{p-1} v(\dfrac{b}{a}) > \dfrac{1}{p} v(\dfrac{c}{a})$, then $v$ is a Gau{\ss} valuation which is uniquely determined. Otherwise, $v$ is not necessarily a Gau{\ss} valuation.
\end{Theorem}

\pars The following corollary is now immediate, which generalizes the findings of [\ref{Gupta Becher ruled residue conics}]:

\begin{Corollary}\label{Coro to main thm}
	Let $K = k(X)(\sqrt[p]{aX^p+c})$ where $p$ is a prime number such that $p\neq\ch k$ and $a\neq 0$. Assume that $k$ contains a primitive $p$-th root of unity. Take a valuation $v$ of $k$ such that $p\neq\ch kv$. Then there exists at most one residue transcendental extension of $v$ to $K$ such that $Kv|kv$ is not ruled. Moreover, if such an extension exists then $v$ is a Gau{\ss} valuation with $vX = \dfrac{1}{p} v(\dfrac{c}{a})$.
\end{Corollary}

\pars In the example of the field $\QQ(\zeta_p)(X)(\sqrt[p]{qX^p+1})$, where $\zeta_p$ is a primitive $p$-th root of unity and $q$ is a prime distinct from $p$, the valuation $v$ corresponding to the $q$-adic valuation of $\QQ$ and its Gau{\ss} extension with $vX = \dfrac{1}{p} v(\dfrac{1}{q}) = - \dfrac{1}{p}$ on $\QQ(X)$ has a ruled residue field over $\FF_q$, according to the following Proposition \ref{Prop characterise non-ruled}, thereby the only potentially residue transcendental non-ruled valuation extension identified in Corollary \ref{Coro to main thm}, may in fact be ruled.


\begin{Proposition}\label{Prop characterise non-ruled}
	Let $(k,v)$ and $K$ be as in Theorem \ref{Thm non ruled necessary}. Further, assume that $p$ is an odd prime and $\dfrac{1}{p-1} v(\dfrac{b}{a}) > \dfrac{1}{p} v(\dfrac{c}{a})$. Take the Gau{\ss} extension of $v$ to $K$ with $vX = \dfrac{1}{p} v(\dfrac{c}{a})$. 
	\newline Then $Kv|kv$ is non-ruled if and only if the following conditions hold true:
	\begin{align*}
		& (i)\,\,  vk+\ZZ\dfrac{1}{p} v(\dfrac{c}{a}) = vk+\ZZ\dfrac{1}{p}vc,\\
		& (ii)\,\,  \dfrac{1}{p} v(\dfrac{c}{a}) - \dfrac{1}{p} vc \notin vk \text{ whenever }\dfrac{1}{p}v(\dfrac{c}{a}) \notin vk.
	\end{align*}
\end{Proposition}

\pars  The problem remains open whether there is a bound on the number of extensions of $v$ to $K$ admitting non-ruled residue field extensions when $\dfrac{1}{p-1} v(\dfrac{b}{a}) \leq \dfrac{1}{p} v(\dfrac{c}{a})$. An initial result in this direction is furnished underneath:

\begin{Proposition}\label{Prop non-ruled non-Gauss}
	Let $(k,v)$ and $K$ be as in Theorem \ref{Thm non ruled necessary}. Assume that $\dfrac{1}{p-1} v(\dfrac{b}{a}) \leq \dfrac{1}{p} v(\dfrac{c}{a})$, $k$ is separable-algebraically closed and $\ch k = \ch kv$. Assume further that $\ch k$ does not divide $p-1$. Then there exist at most two distinct extensions of $v$ to $K$ such that the residue field extension is non-ruled. Moreover, if there are two such extensions, then one of them is necessarily the Gau{\ss} extension with $vX = \dfrac{1}{p-1} v(\dfrac{b}{a})$.
\end{Proposition} 

\pars The converse to Theorem \ref{Thm non ruled necessary} is not necessarily true, as we will see examples of ruled residue field extensions in Section \ref{Sect egs} with $vX = \min  \mathlarger{\mathlarger{\mathlarger{\{}}} \dfrac{1}{p-1} v(\dfrac{b}{a}) , \dfrac{1}{p} v(\dfrac{c}{a}) \mathlarger{\mathlarger{\mathlarger{\}}}}$. For any odd prime $p$, we also provide examples of non-ruled residue field extensions when the valuation is not a Gau{\ss} valuation. 


\section{Initial results}\label{Sect initial results}

We state some preliminary results regarding non-ruled extensions of prime degree. First, we need some preparation. 

\begin{Lemma}\label{Lemma K rel alg closed in L}
	Assume that $L|K$ is an extension of fields such that $K$ is relatively algebraically closed in $L$. Take $a\in\overline{L}$ which is algebraic over $K$. Then $K(a)$ and $L$ are linearly disjoint over $K$.
\end{Lemma}

\begin{proof}
	Take the minimal polynomial $Q(X)$ of $a$ over $L$. Write 
	\[ Q(X) = (X-a)(X-a_2) \dotsc (X-a_n) = \sum_{i=0}^{n} c_i X^i. \]
	Take the minimal polynomial $f(X)$ of $a$ over $K$. Then $Q$ divides $f$ over $L$ and hence each root of $Q$ is also a root of $f$. Thus each $a_i$ is a $K$-conjugate of $a$ and consequently $a_i \in \overline{K}$. Further, observe that each coefficient $c_j$ is a symmetric expression in the roots $a_i$ and hence $c_j \in \overline{K}$ for all $j$. Consequently, $c_j \in \overline{K}\sect L = K$, that is, $Q(X)\in K[X]$. It follows that 
	\[ [K(a):K] = [L(a):L], \]
	that is, $K(a)$ and $L$ are linearly disjoint over $K$. 
\end{proof}

\begin{Proposition}\label{Prop non ruled necessary condns}
	Let $K|k$ be an algebraic function field of dimension one. Take $X\in K$ which is transcendental over $k$ and assume that $K|k(X)$ is a Galois extension of prime degree $p$. Take a valuation $v$ of $K$ and extend it to $\overline{K}$. Assume that $(K|k,v)$ is residue transcendental and $Kv|kv$ is non-ruled. Then the following statements hold true: 
	\sn (i) $[Kv:k(X)v]=p$,
	\n (ii) $\overline{kv}\sect k(X)v = \overline{kv}\sect Kv$,
	\n (iii) $\overline{k}\sect K=k$.
\end{Proposition}

\begin{proof}
	It is well-known that $K|k(X)$ being a Galois extension implies that $[Kv:k(X)v]$ divides $[K:k(X)]$ [\ref{ZS2}, Ch. VI, \S 12, Corollary to Theorem 25]. If $Kv = k(X)v$, then $Kv|kv$ is ruled by the Ruled Residue Theorem, contradicting our assumption. It follows that $Kv\neq k(X)v$ and hence $[Kv:k(X)v] > 1$. We thus have the first assertion.  
	
	\pars Set $F:= \overline{kv}\sect k(X)v$. Then $F|kv$ is a finite extension and $k(X)v = F(Z)$ where $Z$ is transcendental over $F$ by the Ruled Residue Theorem. Suppose that $\th\in \overline{kv}\sect Kv$ such that $\th\notin F$. By definition, $F$ is relatively algebraically closed in $k(X)v$. It then follows from Lemma \ref{Lemma K rel alg closed in L} that $[k(X)v \, (\th) : k(X)v] = [F(\th):F] > 1$. Since $k(X)v\,  (\th) \subseteq Kv$, we conclude from part (i) that $Kv = k(X)v \, (\th)$. It follows that $Kv = F(\th)(Z)$. Now, $Z$ being transcendental over $F$ remains transcendental over $F(\th)$. Moreover, $F(\th)|kv$ is a finite extension. This contradicts the assumption that $Kv|kv$ is non-ruled. We thus have the assertion (ii).
	
	\pars We now suppose that $\th^\prime\in \overline{k}\sect K$ such that $\th^\prime\notin k$. Similar arguments as above imply that $K = k(\th^\prime)(X)$. Hence $K|k$ is a ruled extension. The Ruled Residue Theorem then implies that $Kv|kv$ is also ruled, thereby yielding a contradiction. We thus have the final assertion. 
\end{proof}

\begin{Proposition}\label{Prop Kv generated by residue}
	Let $K|k$ be an algebraic function field such that $K = k(X)(\sqrt[p]{F(X)})$ where $X$ is transcendental over $k$, $p$ is a prime and $F(X)\in k(X)$. Let $v$ be a valuation on $K$ such that $(K|k,v)$ is residue transcendental. Assume that $vF=0$, $p\neq\ch kv$ and $[Kv: k(X)v] = p$. Then $Kv = k(X)v \, (\sqrt[p]{F(X)v})$. 
\end{Proposition}

\begin{proof}
	Fix an extension of $v$ to $\overline{K}$ and denote the henselization of any subfield $L$ of $K$ by $L^h$. Observe that $K^h = k(X)^h.K = k(X)^h (\sqrt[p]{F(X)})$ [\ref{Kuh vln model}, Theorem 5.15]. Hence $[K^h:k(X)^h] \leq p$. Since passing to the henselization leaves the residue fields unchanged, it follows from the condition $[Kv: k(X)v] = p$ and the Fundamental Inequality that
	\begin{equation}\label{eqn 0}
		[K^h:k(X)^h] = [k(X)^h (\sqrt[p]{F(X)}): k(X)^h] = p.
	\end{equation}
	Suppose that $\sqrt[p]{F(X)v} \in k(X)v$. Then $T^p - F(X)v$ has a root in $k(X)v$. The assumption that $p\neq\ch kv$ implies that $T^p - F(X)v$ is a separable polynomial and hence each root is simple. Then $k(X)^h$ being henselian implies that $T^p - F(X)$ has a simple root in $k(X)^h$, whence $\sqrt[p]{F(X)} \in k(X)^h$, thereby contradicting (\ref{eqn 0}). We thus have the proposition.
\end{proof}

\begin{Lemma}\label{Lemma Fv poly over kv}
	Let notations and assumptions be as in Proposition \ref{Prop Kv generated by residue}. Take a minimal pair of definition $(\a,\g)$ for $(k(X)|k,v)$. Assume that $\a\in k$. Set $e$ to be the order of $\g$ modulo $vk$. Take $d\in k$ such that $vd=-e\g$. Further, assume that $F(X)\in k[X]$. Then $F(X)v \in kv\,[d(X-\a)^e v]$.
\end{Lemma}

\begin{proof}
	Without any loss of generality we assume that $\a=0$. Write $F(X)= \sum_{i=0}^{n} c_i X^i$ where $c_i\in k$. Then $0=vF = \min\{ vc_i +i\g \}$. Take some $i$ such that $vc_i + i\g = 0$. Then $i\g\in vk$. The minimality of $e$ implies that $e$ divides $i$. In other words, $vc_i+i\g > 0$ whenever $e$ does not divide $i$. Thus,
	\begin{align*}
		F(X)v &= (\sum_{i=0}^{n} c_i X^i)v =  \sum_{i=0}^{n} (c_i X^i)v  = \sum_{i=0}^{n^\prime} (c_{ie}X^{ie})v\\
		&= \sum_{i=0}^{n^\prime} (\dfrac{c_{ie}}{d^i} d^iX^{ie})v = \sum_{i=0}^{n^\prime} (\dfrac{c_{ie}}{d^i}v) (dX^e v)^i \in kv[dX^e v].
	\end{align*}
\end{proof}

\pars \textbf{For the rest of this paper}, we fix a prime $p\neq \ch k$ and assume that $k$ contains a primitive $p$-th root of unity. Set $K:= k(X)(\sqrt[p]{aX^p + bX +c})$ where $a,b,c \in k$ and $a\neq 0$. Observe that $K|k(X)$ is a Galois extension of degree $p$ by [\ref{Lang}, Ch. VI, \S 6, Theorem 6.2(ii)]. Take a valuation $v$ of $K$ such that $(K|k,v)$ is residue transcendental and $p\neq\ch kv$. Fix an extension of $v$ to $\overline{K}$. Take a minimal pair of definition $(\a,\g)\in \overline{k}\times v\overline{k}$ for $(k(X)|k,v)$. Set $e$ to be the order of $\g$ modulo $vk(\a)$. Since $(\a,\g)$ is also a minimal pair of definition for $(k(\a,X)|k(\a),v)$, it follows from [\ref{APZ characterization of residual trans extns}, Theorem 2.1] that $e=(vk(\a,X):vk(\a))$. As a consequence, we obtain that
\begin{equation}\label{eqn *}
	evk(X) \subseteq vk(\a).
\end{equation}

\begin{Lemma}\label{Lemma (K.F)v:Fv=p}
	Assume that $F$ is any algebraic extension of $k$ such that $\a\in F$ and $Fv = k(\a)v$. Moreover, assume that $Kv|kv$ is not ruled. Then 
	\[ [(K.F)v:F(X)v] = [Kv:k(X)v] =p. \]
\end{Lemma}

\begin{proof}
	It follows from Proposition \ref{Prop non ruled necessary condns} that $[Kv:k(X)v] = p= [K:k(X)]$. Hence $vK= vk(X)$ by the Fundamental Inequality. The condition that $p\neq\ch kv$ implies that $Kv|k(X)v$ is a separable extension. It now follows from [\ref{Dutta Kuh abhyankars lemma}, Theorem 3(2)] that 
	\[ (K.F)v = Kv. F(X)v. \]
	The fact that $\a\in F$ implies that $(\a,\g)$ is also a minimal pair of definition for $(F(X)|F,v)$. Then $F(X)v = Fv(Z) = k(\a)v\, (Z)$ for some $Z$ transcendental over $k(\a)v$ by [\ref{APZ characterization of residual trans extns}, Corollary 2.3]. It further follows from [\ref{APZ characterization of residual trans extns}, Theorem 2.1] that $k(\a)v\subset Kv$. If $Kv\subseteq F(X)v$, then by L$\ddot{u}$roth's Theorem $Kv|k(\a)v$ is also a simple transcendental extension, thereby contradicting the non-ruled hypothesis of $Kv|kv$. We thus have the chain
	\[ k(X)v \subseteq Kv\sect F(X)v \subsetneq Kv. \]
	The condition $[Kv:k(X)v] = p$ then implies that $Kv\sect F(X)v = k(X)v$. Now, $Kv|k(X)v$ is a normal extension and hence Galois by the earlier separability assertion [\ref{ZS2}, Ch. VI, \S 12, Theorem 21]. It follows that $Kv$ and $F(X)v$ are linearly disjoint over $k(X)v$. Thus,
	\[  [(K.F)v:F(X)v] = [Kv. F(X)v:F(X)v] = [Kv:k(X)v] = p. \]
\end{proof}


\section{The case when $e$ and $p$ are coprime}\label{Sect e p coprime}

\begin{Proposition}\label{Prop residue field non ruled min pair}
	Assume that $e$ and $p$ are coprime. Further, assume that $Kv|kv$ is not ruled. Then $K(\a)v|k(\a)v$ is also not ruled. 
\end{Proposition}

\begin{proof}
	If $\a\in k$ then the assertion is vacuously true. We thus assume that $\a\notin k$. As a consequence, 
	\[ vX = v\a < \g = vY \text{ where } Y:= X-\a. \]
	Take $d\in k(\a)$ such that $vd = -e\g$, that is, $vdY^e = 0$. Then $Z:= dY^e v$ is transcendental over $k(\a)v$, and $k(\a,X)v = k(\a)v\, (Z)$ by [\ref{APZ characterization of residual trans extns}, Corollary 2.3]. It follows from Lemma \ref{Lemma (K.F)v:Fv=p} that
	\[  [K(\a)v:k(\a,X)v] = [Kv:k(X)v] = p. \]
	As a consequence, $vK(\a) = vk(\a,X)$. It has been observed in [\ref{APZ characterization of residual trans extns}, Theorem 2.1] that $(vk(\a,X):vk(\a)) = e$. Then $\dfrac{e}{p} v(aX^p + bX +c) \in vk(\a)$. Take $t\in k(\a)$ such that $vt = - \dfrac{e}{p} v(aX^p + bX +c)$, that is,
	\[  vt^p (aX^p + bX+c)^e = 0. \]
	The fact that $e$ and $p$ are coprime implies that $k(\a,X)(\sqrt[p]{aX^p+bX+c}) = k(\a,X)(\sqrt[p]{(aX^p+bX+c)^e})$. As a consequence,
	\[ K(\a) = k(\a,X)(t\sqrt[p]{(aX^p+bX+c)^e})   =  k(\a,X)(\sqrt[p]{t^p (aX^p+bX+c)^e}). \]
	It now follows from Proposition \ref{Prop Kv generated by residue} that 
	\begin{align*}
		K(\a)v &= k(\a,X)v \, (\sqrt[p]{t^p (aX^p+bX+c)^e v}).
	\end{align*} 
	\pars Suppose that $k(\a)v$ is not relatively algebraically closed in $K(\a)v$. Take $\th\in \overline{k(\a)v}\sect K(\a)v$ such that $\th\notin k(\a)v$. Recall that $k(\a)v$ is relatively algebraically closed in $k(\a,X)v$. It follows from Lemma \ref{Lemma K rel alg closed in L} that $k(\a,X)v$ and $k(\a)v \, (\th)$ are linearly disjoint over $k(\a)v$. As a consequence, $[k(\a,X)v \, (\th): k(\a,X)v] > 1$. The facts that $\th\in K(\a)v$ and $[K(\a)v: k(\a,X)v]=p$ then imply that $K(\a)v = k(\a,X)v \, (\th)$. Thus,
	\[ K(\a)v = k(\a)v \, (\th,Z). \]
	Recall that $Z$ is transcendental over $k(\a)v$ and hence over $k(\a)v \, (\th)$. The fact that $\sqrt[p]{t^p (aX^p+bX+c)^e v}$ is contained in $K(\a)v$ implies that we have an expression $\sqrt[p]{t^p (aX^p+bX+c)^e v} = \dfrac{F(Z)}{G(Z)}$ where $F(Z)$ and $G(Z)$ are coprime polynomials over $k(\a)v \, (\th)$. Consequently, 
	\[ G(Z)^p (t^p (aX^p+bX+c)^e v) = F(Z)^p. \]
	It follows from Lemma \ref{Lemma Fv poly over kv} that $t^p (aX^p+bX+c)^e v \in k(\a)v \, [Z] \subseteq k(\a)v \, (\th)[Z]$. Hence $G(Z)^p$ divides $F(Z)^p$ over $k(\a)v \, (\th)$. The coprimality of $F(Z)$ and $G(Z)$ then implies that $G(Z)$ is a unit. After renaming, we obtain an expression
	\[ t^p (aX^p+bX+c)^e v = F(Z)^p \text{ where } F(Z)\in k(\a)v \, (\th)[Z]. \]
	Observe that $\deg_Y t^p (aX^p+bX+c)^e = \deg_X t^p (aX^p+bX+c)^e = pe$. Write $t^p (aX^p+bX+c)^e = \sum_{i=0}^{pe}C_i Y^i$. It follows from the proof of Lemma \ref{Lemma Fv poly over kv} that 
	\[ t^p (aX^p+bX+c)^ev = \sum_{\l=0}^{p} C_{\l e}Y^{\l e}v = \sum_{\l=0}^{p}\dfrac{C_{\l e}}{d^\l} v \, Z^\l.  \]
	Hence $\deg _Z (t^p (aX^p+bX+c)^ev) \leq p$. It follows that either $\deg _Z F(Z) = 1$ or $\deg _Z F(Z)=0$, that is, either $\deg _Z (t^p (aX^p+bX+c)^ev) = p$ or $\deg_Z (t^p (aX^p+bX+c)^ev) = 0$. 
	
	\pars We now write 
    \begin{align*}
		aX^p &= \sum_{i=0}^{p} a \binom{p}{i} \a^i Y^{p-i} = \sum_{i=0}^{p-2} a^\prime_i Y^{p-i} + ap\a^{p-1}Y + a\a^p,\\
		bX&= bY + b\a,\\
		aX^p + bX + c &=  \sum_{i=0}^{p-2} a^\prime_i Y^{p-i} + (ap\a^{p-1}+b)Y + (a\a^p + b\a +c)   =  \sum_{i=0}^{p-2} a^\prime_i Y^{p-i} + b^\prime Y + c^\prime. 
	\end{align*}
    By the multinomial theorem, we have that
    \[  t^p (aX^p+bX+c)^e = \sum_{k_0 + \dotsc + k_p = e}  t^p \binom{e}{k_0, \dotsc, k_p} (a^\prime_0 Y^p)^{k_0} \dotsc (c^\prime)^{k_p}, \]
    where 
    \[ \binom{e}{k_0, \dotsc, k_p} := \dfrac{e!}{(k_0)!\dotsc (k_p)!}. \]
	We have observed in the proof of Lemma \ref{Lemma Fv poly over kv} that 
	\[  t^p (aX^p+bX+c)^ev = \sum_{k_0 + \dotsc + k_p = e}  \mathlarger{\mathlarger{\mathlarger{\mathlarger{\mathlarger{(}}}}}  t^p \binom{e}{k_0, \dotsc, k_p} (a^\prime_0 Y^p)^{k_0} \dotsc (c^\prime)^{k_p} \mathlarger{\mathlarger{\mathlarger{\mathlarger{\mathlarger{)}}}}} v. \]
	The leading term of $t^p (aX^p+bX+c)^ev$ is thus given by $t^p (a^\prime_0 Y^p)^e v$. Now, the fact that $(\a,\g)$ is a minimal pair of definition for $(k(\a,X)|k(\a),v)$ implies that 
	\[ v(aX^p + bX + c) = \min \{ v(a^\prime_i Y^{p-i}), v(b^\prime Y), vc^\prime \}. \]
	We observe that $a^\prime _0 = a$ and $a^\prime_1 = ap\a$. The facts that $p\neq\ch kv$ and $vY= \g > v\a$ imply that
	\[ v(a^\prime_0 Y^p) = va + p\g > va + v\a + (p-1)\g = v(a^\prime_1 Y^{p-1}). \]
	As a consequence, 
	\[ vt^p (a^\prime_0 Y^p)^e > vt^p(a^\prime_1 Y^{p-1})^e \geq v t^p (aX^p+bX+c)^e=0. \]
	It follows that $t^p (a^\prime_0Y^p)^ev = 0$ and hence $\deg_Z (t^p (aX^p+bX+c)^ev) < p$. From our previous discussions, we conclude that $\deg_Z (t^p (aX^p+bX+c)^ev)=0$. The constant term of $t^p (aX^p+bX+c)^ev$ is given by $t^p(c^\prime)^ev$. It follows that
	\begin{equation}\label{eqn 1}
		t^p (aX^p+bX+c)^ev = t^p(c^\prime)^ev \in k(\a)v.
	\end{equation}

    \pars By [\ref{APZ characterization of residual trans extns}, Theorem 2.1] we can take $f(X)\in k[X]$ with $\deg_X f < [k(\a):k]$ such that $vf = vt$, that is, $vf^p(aX^p+bX+c)^e= 0$. Observe that $K = k(X)(\sqrt[p]{f^p(aX^p+bX+c)^e})$. Then $Kv = k(X)v(\sqrt[p]{f^p(aX^p+bX+c)^ev})$ by Proposition \ref{Prop Kv generated by residue}. It follows from [\ref{Dutta min fields implicit const fields}, Lemma 6.1] that $\dfrac{f}{t} v \in k(\a)v$. As a consequence, 
    \begin{align*}
    	f^p(aX^p+bX+c)^ev &= (\dfrac{f}{t}v)^p t^p(aX^p+bX+c)^ev \in k(\a)v,
    \end{align*}
    where the last assertion holds true in light of (\ref{eqn 1}). Set $\d:= f^p(aX^p+bX+c)^ev$. It follows from [\ref{APZ characterization of residual trans extns}, Corollary 2.3] that $k(X)v = k(\a)v \, (Z^\prime)$ where $Z^\prime$ is transcendental over $k(\a)v$. Our observations then imply that
    \[ Kv = k(\a)v \, (\sqrt[p]{\d}, Z^\prime) \]
    which contradicts the assumption that $Kv|kv$ is non-ruled. It follows that $k(\a)v$ is relatively algebraically closed in $K(\a)v$.
    
    \pars Suppose that $K(\a)v|k(\a)v$ is ruled. Then $k(\a)v$ being relatively algebraically closed in $K(\a)v$ implies that $K(\a)v$ is a simple transcendental extension of $k(\a)v$. Since $k(\a)v \subsetneq Kv\subseteq K(\a)v$, L$\ddot{u}$roth's Theorem implies that $Kv|k(\a)v$ is also a simple transcendental extension, which again contradicts the non-ruled hypothesis of $Kv|kv$. We thus have the proposition. 
\end{proof}

\begin{Theorem}\label{Thm e and p coprime}
	Assume that $e$ and $p$ are coprime. Further, assume that $Kv|kv$ is not ruled. Then,
	\[ v(aX^p+bX+c) \geq v(aX^p) = \min \{ v(bX),vc \}.  \]
	Equality holds if and only if $v$ is a Gau{\ss} valuation. Otherwise, we have 
	\[ v(aX^p+bX+c) > v(aX^p) = v(bX). \] 
\end{Theorem}

\begin{proof}
	In light of Proposition \ref{Prop residue field non ruled min pair} we can assume that $\a\in k$ without any loss of generality. Set $Y:= X-\a$ and write $aX^p+bX+c = \sum_{i=0}^{p-2} a^\prime_i Y^{p-i} + b^\prime Y +c^\prime$. 
	
	\pars First, we assume that $v$ is not a Gau{\ss} valuation, that is, $vX=v\a<\g = vY$. Suppose that $\min\{ v(a^\prime_i Y^{p-i}) \} > \min\{ v(b^\prime Y), vc^\prime  \}$. Then $v(aX^p+bX+c) = \min\{ v(a^\prime_i Y^{p-i}), v(b^\prime Y), vc^\prime \} = \min\{ v(b^\prime Y),vc^\prime \}$. Take $t\in k$ such that $vt^p(aX^p+bX+c)^e=0$ and take non-negative integers $k_0,\dotsc,k_p$ satisfying $k_0+\dotsc+k_p=e$. We first assume that $vc^\prime \leq v(b^\prime Y)$. Then $vc^\prime < v(a^\prime_{p-i}Y^i)$ for all $i$. It follows that $v(c^\prime)^{k_{p-i}} \leq v(a^\prime_{p-i}Y^i)^{k_{p-i}}$ and the inequality is strict whenever $k_{p-i}>0$. Moreover, $v(c^\prime)^{k_{p-1}}\leq v(b^\prime Y)^{k_{p-1}}$. Thus
	\[ v t^p (a^\prime_0 Y^p)^{k_0}\dotsc (c^\prime)^{k_p} \geq v t^p (c^\prime)^{k_0}\dotsc (c^\prime)^{k_p} = v(c^\prime)^e,  \] 
	and the inequality is strict whenever $k_{p-i}>0$ for some $i>1$. Employing the same arguments when $v(b^\prime Y)\leq vc^\prime$, we conclude that
	\begin{align*}
		0 &= vt^p (aX^p+b^p+c)^e = \min\{ vt^p(b^\prime Y)^e, vt^p (c^\prime)^e  \}\\
		&\leq v t^p (a^\prime_0 Y^p)^{k_0}\dotsc (c^\prime)^{k_p},
	\end{align*}
where the inequality is strict whenever $k_{p-i}>0$ for some $i>1$. Consider the expansion 
\[  t^p (aX^p+bX+c)^e = \sum_{k_0 + \dotsc + k_p = e}  t^p \binom{e}{k_0, \dotsc, k_p} (a^\prime_0 Y^p)^{k_0} \dotsc (c^\prime)^{k_p}. \]
Each monomial $t^p \binom{e}{k_0, \dotsc, k_p} (a^\prime_0 Y^p)^{k_0} \dotsc (c^\prime)^{k_p}$ contributing to the term corresponding to $Y^{\l e}$ satisfies the additional condition $pk_0 + \dotsc + 2k_{p-2}+k_{p-1} = \l e$. Since $0\leq k_i \leq e$, it follows that $k_{p-i}>0$ for some $i>1$ whenever $\l>1$. Writing $t^p(aX^p+bX+c)^e = \sum_{i=0}^{pe} C_i Y^i$, we then obtain from our previous discussions that
\[ vC_{\l e} > 0 \text{ whenever } \l > 1. \]
It follows from the proof of Lemma \ref{Lemma Fv poly over kv} that
\[ t^p(aX^p+bX+c)^e v = \sum_{\l=0}^{p} C_{\l e}Y^{\l e}v. \]
Since $C_{\l e}v = 0$ whenever $\l>1$, we obtain that
\begin{align*}
	t^p(aX^p+bX+c)^e v &= C_eY^e v + C_0 v = t^p(b^\prime Y)^e v + t^p (c^\prime)^e v \\
	&= \dfrac{t^p (b^\prime)^e}{d}v Z + t^p (c^\prime)^e v\\
	&= \b_1 Z + \b_2,
\end{align*} 
where $d\in k$ such that $vd=-e\g$, $Z:= dY^e v$, $\b_1 := \dfrac{t^p (b^\prime)^e}{d}v \in kv$ and $\b_2 := t^p (c^\prime)^e v\in kv$. It now follows from Proposition \ref{Prop Kv generated by residue} that 
\[ Kv = k(X)v \,(\sqrt[p]{\b_1 Z + \b_2}) = kv \,(Z, \sqrt[p]{\b_1 Z + \b_2}). \]
Thus $Kv = kv\,(\sqrt[p]{\b_1 Z + \b_2})$ if $\b_1 \neq 0$ and $Kv = kv \,(\sqrt[p]{\b_2},Z)$ if $\b_1  = 0$. Both scenarios contradict the assumption that $Kv|kv$ is not ruled. We have thus shown that
\begin{equation}\label{eqn 2}
	\min\{ v(a^\prime_i Y^{p-i}) \} \leq \min\{ v(b^\prime Y), vc^\prime  \}.
\end{equation}
As a consequence, 
\[ v(aX^p+bX+c) = \min \{ v(a^\prime_i Y^{p-i}) \}. \]
Observe that 
\[ v(a^\prime_i Y^{p-i})  \geq va + iv\a + (p-i)\g > v(ap\a^{p-1}Y) > v(a\a^p) \text{ for all } i\in\{0,\dotsc,p-2\}. \]
It follows that 
\[ v(aX^p+bX+c) > v(a\a^p) = v(aX^p). \]
Suppose that $v(aX^p)\neq v(bX)$. Then $v(a\a^{p-1}) = v(aX^{p-1}) \neq vb$. Consequently, $v(ap\a^{p-1}) = v(a\a^{p-1})\neq vb $. It then follows from the triangle inequality that $vb^\prime=v(ap\a^{p-1}+b) = \min\{v(ap\a^{p-1}) ,vb \} $. As a consequence, 
\[ v(a^\prime_i Y^{p-i}) > v(ap\a^{p-1}Y) \geq v(b^\prime Y) \text{ for all } i=0,\dotsc,p-2, \]
which contradicts (\ref{eqn 2}). Thus
\[ v(aX^p) = v(bX). \]
Suppose that $v(aX^p) = v(bX) > vc$. Then $v(a\a^p) = v(b\a) > vc$. By the triangle inequality, we obtain that $vc^\prime = v(a\a^p + b\a +c) = vc$. Thus $v(a^\prime_i Y^{p-i}) > v(a\a^p) > vc = vc^\prime$ for all $i=0,\dotsc,p-2$, which again contradicts (\ref{eqn 2}). We have thus arrived at the following relations:
\[ v(aX^p+bX+c) > v(aX^p) = v(bX) = \min\{ v(bX),vc \}. \]

\pars We now assume that $v$ is a Gau{\ss} valuation. Then $v(aX^p+bX+c) = \min\{ v(aX^p),v(bX),vc \}\leq v(aX^p)$ by definition. Write 
\[ t^p(aX^p+bX+c)^e = \sum_{k_0 + k_1 + k_2 = e}  t^p \binom{e}{k_0, k_1, k_2} (aX^p)^{k_0}(bX)^{k_1}c^{k_2}. \]
Suppose that $v(aX^p) > \min\{ v(bX), vc \}$. Then $v(aX^p+bX+c) = \min\{ v(bX), vc \} < v(aX^p)$. As a consequence, 
\begin{align*}
	0 = v t^p (aX^p+bX+c)^e  &= \min\{ vt^p (bX)^e, vt^pc^e \} \\
	& < vt^p (aX^p)^{k_0}(bX)^{k_1}c^{k_2}\\
	&\leq vt^p\binom{e}{k_0, k_1, k_2} (aX^p)^{k_0}(bX)^{k_1}c^{k_2},
\end{align*}
whenever $k_0+k_1+k_2 = e$ and $k_0>0$. It follows that
\[ t^p (aX^p+bX+c)^e v = \sum_{k_1 + k_2 = e}  t^p \binom{e}{k_1, k_2} (bX)^{k_1}c^{k_2}v. \]
Take $d\in k$ such that $vdX^e = 0$ and set $Z:= dX^ev$. Since $t^p(aX^p+bX+c)^ev \in kv[Z]$ by Lemma \ref{Lemma Fv poly over kv}, we obtain from the above expression that 
\[ t^p (aX^p+bX+c)^e v = t^p (bX)^e v + t^pc^ev = \b^\prime_1 Z + \b^\prime_2,  \]
where $\b^\prime_1:= \dfrac{t^p b^e}{d}v\in kv$ and $\b^\prime_2:= t^pc^ev\in kv$. It now follows from Proposition \ref{Prop Kv generated by residue} that 
\[ Kv = k(X)v \,(\sqrt[p]{\b^\prime_1 Z + \b^\prime_2}) = kv \,(Z, \sqrt[p]{\b^\prime_1 Z + \b^\prime_2}), \]
where the final equality follows from [\ref{APZ characterization of residual trans extns}, Corollary 2.3]. Thus $Kv = kv\,(\sqrt[p]{\b^\prime_1 Z + \b^\prime_2})$ if $\b^\prime_1 \neq 0$ and $Kv = kv \,(\sqrt[p]{\b^\prime_2},Z)$ if $\b^\prime_1  = 0$. Since $Kv|kv$ is non-ruled, we have arrived at a contradiction. It follows that $v(aX^p) \leq \min\{ v(bX),vc \}$ and hence 
\[ v(aX^p+bX+c) = v(aX^p). \]
Now suppose that $v(aX^p) < \min\{ v(bX),vc \}$. Then
\[  0 = vt^p (aX^p+bX+c)^e = vt^p(aX^p)^e < vt^p \binom{e}{k_0, k_1, k_2} (aX^p)^{k_0}(bX)^{k_1}c^{k_2} \]
whenever $k_0+k_1+k_2 = e$ and $k_0 < e$. It follows that
\[ t^p(aX^p+bX+c)^e v = t^p (aX^p)^e v = \dfrac{t^p a^e}{d^p}v (dX^ev)^p = \b Z^p, \]
where $\b:= \dfrac{t^p a^e}{d^p}v \in kv$. It then follows from Proposition \ref{Prop Kv generated by residue} that $Kv = k(X)v\, (\sqrt[p]{\b Z^p}) = kv\, (\sqrt[p]{\b},Z)$, which again contradicts the non-ruled hypothesis. We have now proved the theorem.
\end{proof}




\section{The case when $p$ divides $e$}\label{Sect p divides e}

At first, we make some general observations. 

\begin{Proposition}\label{Prop p divides e}
	Assume that $p$ divides $e$. Further, assume that $[Kv: k(X)v] = p$. Write $aX^p+bX+c = \sum_{i=0}^{p-2} a^\prime_i Y^{p-i} + b^\prime Y +c^\prime$, where $Y:= X-\a$. Then,
	\[ v(aX^p+bX+c) =\min \{ v(a^\prime_0 Y^p), vc^\prime \}< \min\{ v(a^\prime_i Y^{p-i}), v(b^\prime Y) \}_{i=1, \dotsc, p-2}.  \]
	Moreover, $v$ is a Gau{\ss} valuation whenever $v(aX^p+bX+c) = v(a^\prime_0 Y^p)$.
\end{Proposition}

\begin{proof}
	Set $n:= \dfrac{e}{p}$. Then $n\in\ZZ_{\geq 1}$. The condition $[Kv:k(X)v]=p=[K:k(X)]$ implies that $vK = vk(X)$. Since $v(aX^p+bX+c) = p v(\sqrt[p]{aX^p+bX+c})$, it follows that 
	\begin{equation}\label{eqn 3}
		nv(aX^p+bX+c) = ev(\sqrt[p]{aX^p+bX+c}) \in vk(\a),
	\end{equation}
	where the containment follows from (\ref{eqn *}). By definition, $v(aX^p+bX+c) = \min \{ v(a^\prime_i Y^{p-i}), v(b^\prime Y), vc^\prime \}$. Suppose that $v(aX^p+bX+c) = v(b^\prime Y)$. It follows from (\ref{eqn 3}) that $nvb^\prime +n\g\in vk(\a)$. $n$ being an integer implies that $(b^\prime)^n\in k(\a)$ and hence $nvb^\prime\in vk(\a)$. As a consequence, $n\g\in vk(\a)$ which contradicts the minimality of $e$. Now, suppose that $v(aX^p+bX+c) = v(a^\prime_i Y^{p-i})$ for some $i\in\{1,\dotsc,p-2\}$. Similar arguments yield that $n(p-i)\g\in vk(\a)$, that is, $e$ divides $n(p-i)$. Consequently, $p$ divides $p-i$, which again leads to a contradiction. We thus have the first assertion of the proposition. 
	\pars Now assume that $v(aX^p+bX+c) = v(a^\prime_0 Y^p)$. We observe from our previous discussions that $v(a^\prime_0 Y^p) < v(a^\prime_1 Y^{p-1})$, that is, 
	\[ v(aY^p) < v(ap\a Y^{p-1}). \]
	The fact that $p\neq\ch kv$ implies that $vp=0$. As a consequence, $\g < v\a$. It now follows from the triangle inequality that $\g = v(X-\a) = vX$ and hence $(0,\g)$ is also a minimal pair of definition for $(k(X)|K,v)$, that is, $v$ is a Gau{\ss} valuation.
\end{proof}

\begin{Theorem}\label{Thm p divides e}
	Assume that $p$ divides $e$. Further, assume that $Kv|kv$ is not ruled. Then $v$ is a Gau{\ss} valuation satisfying 
	\[ v(aX^p+bX+c) = v(aX^p) = vc < v(bX). \]
\end{Theorem}

\begin{proof}
	Suppose that $v$ is not a Gau{\ss} valuation, that is, $vX=v\a<\g=vY$. It follows from Proposition \ref{Prop p divides e} that
	\begin{equation}\label{eqn 4}
		v(aX^p+bX+c) =vc^\prime < \min\{ v(a^\prime_i Y^{p-i}), v(b^\prime Y) \}.
	\end{equation}
	As a consequence, we obtain that
	\begin{equation}\label{eqn 5}
		\dfrac{aX^p+bX+c}{c^\prime}v = 1v.
	\end{equation}

    \pars Suppose that $\dfrac{1}{p}vc^\prime\in vk(\a)$. Take $t\in k(\a)$ such that $vt = -\dfrac{1}{p} vc^\prime$. By [\ref{APZ characterization of residual trans extns}, Theorem 2.1] we can take $f(X)\in k[X]$ with $\deg_X f < [k(\a):k]$ such that $vf=vt$. It follows from [\ref{Dutta min fields implicit const fields}, Lemma 6.1] that $\dfrac{f}{t}v\in k(\a)v$. Hence, 
    \[ k(\a)v \, (\sqrt[p]{f^p c^\prime v}) = k(\a)v \, (\sqrt[p]{t^p c^\prime v}).  \]
    Observe that $K = k(X)(\sqrt[p]{f^p (aX^p+bX+c)})$. We have $[Kv:k(X)v]=p$ by Proposition \ref{Prop non ruled necessary condns}. It now follows from Proposition \ref{Prop Kv generated by residue} that $Kv = k(X)v \, (\sqrt[p]{f^p(aX^p+bX+c)}v)$. Expanding $aX^p+bX+c$ in terms of $Y$, we obtain from (\ref{eqn 4}) that $f^p(aX^p+bX+c)v = f^pc^\prime v$. Moreover, it follows from [\ref{APZ characterization of residual trans extns}, Corollary 2.3] that $k(X)v = k(\a)v \, (Z^\prime)$   where $Z^\prime$ is transcendental over $k(a)v$. Thus,
    \begin{align*}
    	Kv&= k(\a)v \, (Z^\prime, \sqrt[p]{f^p (aX^p+bX+c)v}) = k(\a)v \, (Z^\prime, \sqrt[p]{f^p c^\prime v})\\
    	&= k(\a)v \, (\sqrt[p]{t^p c^\prime v})(Z^\prime).
    \end{align*}
  However, this contradicts the non-ruled hypothesis of $Kv|kv$. It follows that 
    \[ \dfrac{1}{p} vc^\prime\notin vk(\a). \]

\pars Consider the valued field extension $(k(\sqrt[p]{c^\prime}, \a)| k(\a),v)$. Observe that $k(\sqrt[p]{c^\prime},\a)|k(\a)$ is a Galois extension by [\ref{Lang}, Theorem 6.2(ii)]. Moreover, $v\sqrt[p]{c^\prime}\notin vk(\a)$ by our prior findings. It now follows from [\ref{ZS2}, Ch. VI, \S 12, Corollary to Theorem 25] that $(vk(\sqrt[p]{c^\prime},\a) : vk(\a)) =p$ and hence 
\[ k(\sqrt[p]{c^\prime},\a)v = k(\a)v. \]
It follows from Lemma \ref{Lemma (K.F)v:Fv=p} that
\begin{equation}\label{eqn 6}
	[K(\sqrt[p]{c^\prime},\a)v: k(\sqrt[p]{c^\prime},\a,X)v] = [Kv:k(X)v]=p.
\end{equation}
The fact that $\a\in k(\sqrt[p]{c^\prime},\a)$ implies that $(\a,\g)$ is also a minimal pair for $( k(\sqrt[p]{c^\prime},\a,X)|k(\sqrt[p]{c^\prime},\a),v)$. It follows from [\ref{APZ characterization of residual trans extns}, Corollary 2.3] that 
\[ k(\sqrt[p]{c^\prime},\a,X)v = k(\sqrt[p]{c^\prime},\a)v\, (Z^{\prime\prime}) = k(\a)v\, (Z^{\prime\prime}), \]
where $Z^{\prime\prime}$ is transcendental over $k(\a)v$. Observe that $K(\sqrt[p]{c^\prime},\a) = k(\sqrt[p]{c^\prime},\a,X)(\sqrt[p]{aX^p+bX+c}) = k(\sqrt[p]{c^\prime},\a,X) \mathlarger{\mathlarger{\mathlarger{(}}}\sqrt[p]{\dfrac{aX^p+bX+c}{c^\prime}} \mathlarger{\mathlarger{\mathlarger{)}}} $. It now follows from Proposition \ref{Prop Kv generated by residue} that 
\[ K(\sqrt[p]{c^\prime},\a)v = k(\sqrt[p]{c^\prime},\a,X)v  \mathlarger{\mathlarger{\mathlarger{(}}}\sqrt[p]{\dfrac{aX^p+bX+c}{c^\prime}}v \mathlarger{\mathlarger{\mathlarger{)}}} . \] 
In light of (\ref{eqn 5}), we obtain that $K(\sqrt[p]{c^\prime},\a)v = k(\sqrt[p]{c^\prime},\a,X)v \, (\zeta)$, where $\zeta $ is a $p$-th root of $1v$. However, the assumption that $k$ contains a primitive $p$-th root of unity implies that $\zeta \in kv$ and hence $K(\sqrt[p]{c^\prime},\a)v = k(\sqrt[p]{c^\prime},\a,X)v$, which contradicts (\ref{eqn 6}). Thus $v$ is a Gau{\ss} valuation. It then follows from Proposition \ref{Prop p divides e} that
\[ v(aX^p+bX+c) = \min\{ v(aX^p), vc \} < v(bX).  \]

\pars Suppose that $vc < v(aX^p)$. Then $v(aX^p+bX+c) = vc < \min\{ v(aX^p),v(bX) \}$. Replacing $c^\prime$ by $c$, $\a$ by $0$ and repeating the same arguments used in the proof so far (with suitable simplifications) then lead us to a contradiction. As a consequence, $v(aX^p) \leq vc$.

\pars Suppose now that $v(aX^p+bX+c) = v(aX^p) < vc$. We have $v(aX^p) < v(bX)$ from Proposition \ref{Prop p divides e}. As a consequence, 
\begin{equation}\label{eqn 7}
	\dfrac{aX^p+bX+c}{aX^p}v = 1v.
\end{equation}
Consider the valued field extension $(k(\sqrt[p]{a})|k,v)$. Suppose that $k(\sqrt[p]{a})v = kv$. It then follows from Lemma \ref{Lemma (K.F)v:Fv=p} that
\[ [K(\sqrt[p]{a})v:k(\sqrt[p]{a},X)v] = p.  \]
Observe that $K(\sqrt[p]{a}) = k(\sqrt[p]{a},X)(\sqrt[p]{aX^p+bX+c}) = k(\sqrt[p]{a},X)\mathlarger{\mathlarger{\mathlarger{(}}}  \sqrt[p]{ \dfrac{aX^p+bX+c}{aX^p} } \mathlarger{\mathlarger{\mathlarger{)}}}$. It now follows from Proposition \ref{Prop Kv generated by residue} that $K(\sqrt[p]{a})v = k(\sqrt[p]{a},X)v \, \mathlarger{\mathlarger{\mathlarger{(}}} \sqrt[p]{ \dfrac{aX^p+bX+c}{aX^p} v } \mathlarger{\mathlarger{\mathlarger{)}}} $. In light of (\ref{eqn 7}) we conclude that $K(\sqrt[p]{a})v = k(\sqrt[p]{a},X)v\, (\zeta)$ where $\zeta$ is a $p$-th root of $1v$, whence it leads to a contradiction. Thus $kv \subsetneq k(\sqrt[p]{a})v$. As a consequence, we obtain that $\sqrt[p]{a}\notin k$ and hence $k(\sqrt[p]{a}) | k$ is a Galois extension of degree $p$. Moreover, 
\[ [k(\sqrt[p]{a})v : kv] = p.  \]
The fact that $p\neq \ch kv$ then implies that $k(\sqrt[p]{a})v|kv$ is a separable and hence primitive extension. Observe that $kv$ is relatively algebraically closed in $k(X)v$ by [\ref{APZ characterization of residual trans extns}, Theorem 2.1]. By Proposition \ref{Prop non ruled necessary condns}, $kv$ is also relatively algebraically closed in $Kv$. It now follows from Lemma \ref{Lemma K rel alg closed in L} that 
\[ p= [k(\sqrt[p]{a})v : kv] = [k(\sqrt[p]{a})v. k(X)v : k(X)v] = [k(\sqrt[p]{a})v. Kv : Kv].   \]
Since $p\neq\ch kv$ and $[k(\sqrt[p]{a})v : kv] = p$, we obtain from [\ref{Dutta Kuh abhyankars lemma}, Theorem 3(2)] that $k(\sqrt[p]{a})v. k(X)v = k(\sqrt[p]{a},X)v$ and $k(\sqrt[p]{a})v. Kv = K(\sqrt[p]{a})v$. The above relations can be thus rewritten as 
\[  p= [k(\sqrt[p]{a})v : kv] = [k(\sqrt[p]{a},X)v: k(X)v] = [K(\sqrt[p]{a})v : Kv]. \]
Moreover, observe that $K(\sqrt[p]{a})v = k(\sqrt[p]{a},X)v.Kv$. We have thus obtained that $Kv$ and $k(\sqrt[p]{a},X)v$ are linearly disjoint over $k(X)v$. It follows that
\begin{equation}\label{eqn 8}
	[K(\sqrt[p]{a})v : k(\sqrt[p]{a},X)v] = [Kv:k(X)v] = p.
\end{equation}
Observe that $K(\sqrt[p]{a}) = k(\sqrt[p]{a},X) (\sqrt[p]{aX^p+bX+c}) = k(\sqrt[p]{a},X) \mathlarger{\mathlarger{\mathlarger{(}}}  \sqrt[p]{ \dfrac{aX^p+bX+c}{aX^p} } \mathlarger{\mathlarger{\mathlarger{)}}}$. It then follows from Proposition \ref{Prop Kv generated by residue} that $K(\sqrt[p]{a}) v = k(\sqrt[p]{a},X)v \mathlarger{\mathlarger{\mathlarger{(}}}  \sqrt[p]{ \dfrac{aX^p+bX+c}{aX^p}v } \mathlarger{\mathlarger{\mathlarger{)}}}$. But this contradicts (\ref{eqn 8}) in light of (\ref{eqn 7}). We have now proved the theorem.
\end{proof}


\section{Proofs of main results}\label{Sect proof of main results}

\subsection{Proof of Theorem \ref{Thm non ruled necessary}}

\begin{proof}
	The following cases are possible in light of Theorems \ref{Thm e and p coprime} and \ref{Thm p divides e}:
	\begin{align*}
		1) \, v(aX^p+bX+c) &= v(aX^p) = vc < v(bX),\\
		2) \, v(aX^p+bX+c) &= v(aX^p) = vc = v(bX),\\
		3) \, v(aX^p+bX+c) &= v(aX^p) = v(bX) < vc,\\
		4) \, v(aX^p+bX+c) &> v(aX^p) = v(bX) \leq vc.
	\end{align*}
We further observe that $v$ is not a Gau{\ss} valuation only in Case 4). In Case 1), the condition $v(aX^p) = vc$ implies that $vX = \dfrac{1}{p} v(\dfrac{c}{a})$. The relation $v(aX^p) < v(bX)$ then implies that $\dfrac{1}{p-1} v(\dfrac{b}{a}) > vX = \dfrac{1}{p} v(\dfrac{c}{a})$. We similarly analyse cases 2) and 3). In Case 4), the condition $v(aX^p) = v(bX)$ implies that $vX = \dfrac{1}{p-1} v(\dfrac{b}{a})$. Finally, the relation $v(aX^p) \leq vc$ leads us to $\dfrac{1}{p-1} v(\dfrac{b}{a}) = vX \leq \dfrac{1}{p} v(\dfrac{c}{a})$. The theorem now follows.
\end{proof}


\subsection{Proof of Proposition \ref{Prop characterise non-ruled}}

%

\begin{Lemma}\label{Lemma non ruled char 0}
	Let $k$ be a field. Take $X$ transcendental over $k$ and an odd prime $p$. Assume that $\ch k \neq p$ and $k$ contains a primitive $p$-th root of unity. Set $K:= k(X)(\sqrt[p]{a_nX^n+a_{n+1}X^{n+1}})$ where $1\leq n\leq p-2$ and $a_n,a_{n+1}\in k\setminus\{0\}$. Then $K|k$ is not a ruled extension.  
\end{Lemma}

\begin{proof}
It follows from [\ref{Stichtenoth book}, Proposition 6.3.1] that $\overline{k}\sect K = k$. Thus $K|k$ is a ruled extension if and only if it is a rational function field over $k$. On the other hand, it follows from [\ref{Stichtenoth book}, Proposition 1.6.3] that a necessary condition for $K|k$ to be a rational extension is that $K|k$ has genus zero. Since $K|k$ has genus $\dfrac{p-1}{2}$ by [\ref{Stichtenoth book}, Proposition 6.3.1], we have the lemma.

\end{proof}

\pars We can now give a \emph{proof of Proposition \ref{Prop characterise non-ruled}}:

\begin{proof}
	The conditions that $v$ is a Gau{\ss} valuation with $vX = \dfrac{1}{p} v(\dfrac{c}{a}) < \dfrac{1}{p-1}v(\dfrac{b}{a})$ imply that
	\begin{equation}\label{eqn 17}
		v(aX^p+bX+c) = v(aX^p) = vc < v(bX).
	\end{equation}
    In particular, 
    \[ \dfrac{1}{p} vc = v(\sqrt[p]{aX^p+bX+c}) \in vK.  \]
    Observe that
    \begin{equation}\label{eqn 18}
    	vk(X) = vk + \ZZ\dfrac{1}{p} v(\dfrac{c}{a}).
    \end{equation}
	\pars We first assume that $Kv|kv$ is not a ruled extension. Then $vK = vk(X)$ and hence $\dfrac{1}{p} vc \in vk(X) $ by part (i) of Proposition \ref{Prop non ruled necessary condns} and the Fundamental Inequality. As a consequence, we obtain from (\ref{eqn 18}) that
	\begin{equation}\label{eqn 19}
		vk + \ZZ\dfrac{1}{p} vc \subseteq vk + \ZZ\dfrac{1}{p} v(\dfrac{c}{a}).
	\end{equation}
	Observe that $(vk + \ZZ\dfrac{1}{p} v(\dfrac{c}{a}) : vk)$ is either $1$ or $p$. The same also holds for $(vk + \ZZ\dfrac{1}{p} vc : vk)$. It follows that the containment in (\ref{eqn 19}) is strict only when $\dfrac{1}{p} vc\in vk$ and $\dfrac{1}{p} v(\dfrac{c}{a})\notin vk$. Assuming strict containment, we have $k(X)v = kv\, (\dfrac{aX^p}{c}v)$ by [\ref{APZ characterization of residual trans extns}, Corollary 2.3]. Set $Z:= \dfrac{aX^p}{c}v$. Take $t\in k$ such that $vt = \dfrac{1}{p} vc$ and write $K = k(X) \mathlarger{\mathlarger{\mathlarger{(}}}  \sqrt[p]{ \dfrac{aX^p+bX+c}{t^p} } \mathlarger{\mathlarger{\mathlarger{)}}}$. Observe that
	\[ \dfrac{aX^p+bX+c}{t^p}v = \dfrac{aX^p}{t^p}v+\dfrac{c}{t^p}v = \b Z+\b, \]
	where $\b:= \dfrac{c}{t^p}v \in kv\setminus\{0\}$. We thus have a chain 
	\[ k(X)v = kv\, (Z)\subsetneq kv\, (Z,\sqrt[p]{\b Z+\b})\subseteq Kv. \]
	The condition $[Kv:k(X)v]=p$ then implies that $Kv = kv\, (Z,\sqrt[p]{\b Z+\b}) = kv\, (\sqrt[p]{\b Z+\b})$ which contradicts the non-ruled assumption. We have thus shown that the containment in (\ref{eqn 19}) is an equality, that is, $(i)$ holds.
	\newline Now assume that $\dfrac{1}{p} v(\dfrac{c}{a}) \notin vk$. Then $(ii)$ follows directly whenever $\dfrac{1}{p} vc \in vk$. Suppose that $\dfrac{1}{p} v(\dfrac{c}{a}), \dfrac{1}{p} vc \notin vk$ but $\dfrac{1}{p} v(\dfrac{c}{a}) - \dfrac{1}{p} vc \in vk$. Then $\dfrac{1}{p}va = \dfrac{1}{p} vc - \dfrac{1}{p} v(\dfrac{c}{a}) \in vk$. Take $d\in k$ such that $vd = \dfrac{1}{p} va$ and write $K = k(X) \mathlarger{\mathlarger{\mathlarger{(}}}  \sqrt[p]{ \dfrac{aX^p+bX+c}{d^pX^p} } \mathlarger{\mathlarger{\mathlarger{)}}}$. Observe that
	\[ \dfrac{aX^p+bX+c}{d^pX^p}v = \b^\prime + \dfrac{\b^\prime}{Z}, \]
	where $\b^\prime:= \dfrac{a}{d^p}v\in kv\setminus \{0\}$. As a consequence we obtain that $Kv = kv\, (Z,\sqrt[p]{\b^\prime + \dfrac{\b^\prime}{Z}}) = kv\, (\sqrt[p]{\b^\prime + \dfrac{\b^\prime}{Z}})$ which again contradicts that $Kv|kv$ is non-ruled. Thus $\dfrac{1}{p} v(\dfrac{c}{a}) - \dfrac{1}{p} vc \notin vk$, that is, $(ii)$ holds.

	\pars Conversely, we assume that $(i)$ and $(ii)$ hold true. We first assume that $\dfrac{1}{p} v(\dfrac{c}{a}) \in vk$. Then $\dfrac{1}{p}vc$ is also contained in $vk$ by $(i)$. Take $t_1, d_1 \in k$ such that $vd_1 = -\dfrac{1}{p}v(\dfrac{c}{a})$ and $vt_1=\dfrac{1}{p}vc$. Then $k(X)v = kv \, (d_1 Xv)$ by [\ref{APZ characterization of residual trans extns}, Corollary 2.3]. Set $Z_1:= d_1Xv$. Write $K = k(X) \mathlarger{\mathlarger{\mathlarger{(}}}  \sqrt[p]{ \dfrac{aX^p+bX+c}{t_1^p} } \mathlarger{\mathlarger{\mathlarger{)}}} $. Observe that
	\[  \dfrac{aX^p+bX+c}{t_1^p}v = \d_1 Z_1^p + \d_2, \]
	where $\d_1:= \dfrac{a}{t_1^p d_1^p}v, \, \d_2:= \dfrac{c}{t_1^p}v \in kv\setminus \{0\}$. As a consequence, we obtain that
	\[ Kv = kv\,(Z_1, \sqrt[p]{\d_1 Z_1^p + \d_2}).  \]
	If $\sqrt[p]{\d_1 Z_1^p + \d_2} \in kv[Z_1]$, then $Kv = kv(Z_1)$ is a simple transcendental extension. Hence $Kv|kv$ has genus zero by [\ref{Stichtenoth book}, Proposition 1.6.3]. However, the genus of $Kv|kv$ is $\dfrac{(p-1)(p-2)}{2}$ by [\ref{Stichtenoth book}, Example 6.3.4]. It follows that $\sqrt[p]{\d_1 Z_1^p + \d_2} \notin kv[Z_1]$. Thus the irreducible decomposition of $\d_1 Z_1^p + \d_2$ in $kv[Z_1]$ is not of the form $g(Z_1)^p$ for some $g(Z_1) \in kv[Z_1]$. Employing [\ref{Stichtenoth book}, Proposition 6.3.1], we then obtain that $kv$ is relatively algebraically closed in $Kv$. Thus $Kv|kv$ is a ruled extension if and only if it is a simple transcendental extension. Since a necessary condition for the latter to hold is that $Kv|kv$ has genus zero [\ref{Stichtenoth book}, Proposition 1.6.3], we conclude that $Kv|kv$ is not a ruled extension.
	\pars We now assume that $\dfrac{1}{p} v(\dfrac{c}{a}) \notin vk$. Then $k(X)v = kv\, (Z)$ where $Z:= \dfrac{aX^p}{c}v$ by [\ref{APZ characterization of residual trans extns}, Corollary 2.3]. In light of $(i)$ we have an expression 
	\[  \dfrac{1}{p} vc = vr + \dfrac{m}{p} v(\dfrac{c}{a}) \text{ where } r\in k, \, 0\leq m \leq p-1.  \]
	Note that $m\neq 0$ by $(i)$ and $m\neq 1$ by $(ii)$, that is, 
	\begin{equation}\label{eqn 22}
		2\leq m\leq p-1.
	\end{equation}  
	Observe that $vc = v(rX^m)^p$. Write $\dfrac{c}{(rX^m)^p}v = \dfrac{\d}{Z^m}$ where $\d:= \dfrac{a^m}{r^p c^{m-1}}v \in kv\setminus\{0\}$. Write $K = k(X)  \mathlarger{\mathlarger{\mathlarger{(}}}  \sqrt[p]{ \dfrac{aX^p+bX+c}{(rX^m)^p} } \mathlarger{\mathlarger{\mathlarger{)}}}$. Observe that
	\[  \dfrac{aX^p+bX+c}{(rX^m)^p} v = \dfrac{aX^p}{(rX^m)^p}v + \dfrac{c}{(rX^m)^p}v = \dfrac{\d}{Z^m}(Z+1). \]
	As a consequence, we obtain that 
	\[  Kv = kv\,  \mathlarger{\mathlarger{\mathlarger{(}}}  Z, \sqrt[p]{\dfrac{\d}{Z^m}(Z+1)}  \mathlarger{\mathlarger{\mathlarger{)}}}  = kv\, (Z, \sqrt[p]{\d Z^{p-m+1} + \d Z^{p-m}}). \]
	It follows from (\ref{eqn 22}) that $1\leq p-m \leq p-2$. Then $Kv|kv$ is not ruled by Lemma \ref{Lemma non ruled char 0}.
\end{proof}


\subsection{Proof of Proposition \ref{Prop non-ruled non-Gauss}}

\begin{Lemma}\label{Lemma non-ruled non-Gauss}
	Assume that $k$ is separable-algebraically closed, $\ch k = \ch kv$ and $\dfrac{1}{p-1} v(\dfrac{b}{a}) \leq \dfrac{1}{p} v(\dfrac{c}{a})$. Further, assume that $Kv|kv$ is not ruled and $v$ is not a Gau{\ss} valuation. Take a minimal pair of definition $(\a,\g)$. Set $Y:= X-\a$ and write 
	\[  aX^p+bX+c = \sum_{i=0}^{p-2} a^\prime_i Y^{p-i} + b^\prime Y + c^\prime. \]
	Set 
	\[  i_0:= \max \mathlarger{\mathlarger{\mathlarger{\{}}} i \mathlarger{\mathlarger{\mathlarger{\mid}}}  1\leq i\leq p-2 \text{ and }  \binom{p}{i}\neq 0 \mathlarger{\mathlarger{\mathlarger{\}}}}.  \]
	Then,
	\[  \g = \min \mathlarger{\mathlarger{\mathlarger{\{}}} \dfrac{1}{p-i_0-1} v \mathlarger{\mathlarger{\mathlarger{ ( }}}  \dfrac{b^\prime \a^{p-i_0-1}}{b} \mathlarger{\mathlarger{\mathlarger{ ) }}} , \dfrac{1}{p-i_0} v \mathlarger{\mathlarger{\mathlarger{ ( }}}  \dfrac{c^\prime \a^{p-i_0-1}}{b} \mathlarger{\mathlarger{\mathlarger{ ) }}}   \mathlarger{\mathlarger{\mathlarger{\}}}}.     \]
	Moreover, we have
	\[ \dfrac{a\a^{p-1}}{b} v = - \dfrac{1}{p} v \text{ and } \dfrac{c}{b\a}v = \dfrac{1-p}{p}v.  \]
\end{Lemma}

\begin{proof}
	The fact that $(k,v)$ is separable-algebraically closed implies that $vk$ is a divisible group [\ref{Engler book}, Theorem 3.2.11]. As a consequence, $\g\in vk$. Moreover, $\a\in k$ by [\ref{Dutta min fields implicit const fields}, Proposition 3.6]. It then follows from Proposition \ref{Prop non ruled necessary condns} (i) and the Fundamental Inequality that
	\[  vk = vk(X) = vK. \]
	Consequently, $\dfrac{1}{p} v(aX^p+bX+c) \in vk$. Take $t\in k$ such that $vt^p(aX^p+bX+c) = 0$. It then follows from Proposition \ref{Prop Kv generated by residue} that $Kv = k(X)v (\sqrt[p]{t^p(aX^p+bX+c)v})$.
	
	\pars Observe that $a^\prime_i = a \binom{p}{i}\a^i$ and hence $a^\prime_i \neq 0$ whenever $\binom{p}{i}\neq 0$. The fact that $\ch k =\ch kv$ implies that $v\binom{p}{i} =0$ whenever $\binom{p}{i} \neq 0$. As a consequence, 
	\[ v a^\prime_i = v(a\a^i) \text{ whenever } a^\prime_i \neq 0.  \]
	Take $j<i$ such that $a^\prime_j , a^\prime_i \neq 0$. Then
	\[  v(a^\prime_j Y^{p-j}) = va + jv\a + (p-j) \g > va + iv\a + (p-i)\g = v(a^\prime_i Y^{p-i}),  \]
	where the inequality follows from the fact that $v$ is not a Gau{\ss} valuation and hence $v\a< \g$. We have thus obtained that
	\[  \min\{ v(a^\prime_i Y^{p-i}) \} = v(a^\prime_{i_0} Y^{p-i_0}) < v(a^\prime_i Y^{p-i}) \text{ for all } i\neq i_0. \]
	It then follows from (\ref{eqn 2}) that 
	\begin{equation}\label{eqn 23}
		v(a^\prime_{i_0} Y^{p-i_0}) \leq \min \{ v(b^\prime Y), vc^\prime \}.
	\end{equation}
    $(\a,\g)$ being a minimal pair of definition implies that
	\[  v(aX^p+bX+c) = \min \{  v(a^\prime_i Y^{p-i}), v(b^\prime Y), vc^\prime  \} =  v(a^\prime_{i_0} Y^{p-i_0}).  \] 
	Suppose that $v(a^\prime_{i_0} Y^{p-i_0}) < \min \{ v(b^\prime Y), vc^\prime \}$. Then 
	\begin{align*}
		0 = vt^p (aX^p+bX+c) & = v t^p (a^\prime_{i_0} Y^{p-i_0}) < v t^p (a^\prime_i Y^{p-i}) \text{ for all } i\neq i_0, \\
		0 = vt^p (aX^p+bX+c) & = v t^p (a^\prime_{i_0} Y^{p-i_0}) < \min\{ v t^p (b^\prime Y), v t^p c^\prime \}.
	\end{align*}
	As a consequence, we obtain that
	\[  t^p (aX^p+bX+c)v = t^pa^\prime_{i_0} Y^{p-i_0} v.  \]
	Take $d\in k$ such that $vd=-\g$. It follows from [\ref{APZ characterization of residual trans extns}, Corollary 2.3] that $k(X)v = kv\, (Z)$ where $Z:= dYv$. Write
	\[  t^p (aX^p+bX+c)v = \b Z^{p-i_0}  \text{ where } \b:= \dfrac{t^p a^\prime_{i_0}}{d^{p-i_0}} \in kv\setminus\{0\}.  \]
	We have thus obtained the chain
	\[  k(X)v = kv\, (Z) \subsetneq kv\, (Z, \sqrt[p]{\b Z^{p-i_0}}) \subseteq Kv. \]
	Since $[Kv:k(X)v] = p$ by Proposition \ref{Prop non ruled necessary condns}, we conclude that $Kv = kv \, (Z, \sqrt[p]{\b Z^{p-i_0}})$. The fact that $k$ is separable-algebraically closed implies that $kv$ is algebraically closed and hence $\sqrt[p]{\b} \in kv$. As a consequence, 
	\[ Kv = kv \, (Z, \sqrt[p]{\b Z^{p-i_0}}) = kv \, (Z, \sqrt[p]{ Z^{p-i_0}}) = kv \, (\sqrt[p]{ Z}), \]
	where the last equality follows from the coprimality of $p$ and $p-i_0$. Since this again contradicts the non-ruled hypothesis, we conclude that the inequality in (\ref{eqn 23}) is actually an equality.

	\pars We first assume that $v(a^\prime_{i_0} Y^{p-i_0}) = v(b^\prime Y) \leq vc^\prime$. The fact that $v$ is not a Gau{\ss} valuation, coupled with Theorem \ref{Thm non ruled necessary} implies that $v\a = vX= \dfrac{1}{p-1} v(\dfrac{b}{a})$ and hence $v(a\a^{p-1}) = vb$. Then,
	\begin{align*}
		v(a^\prime_{i_0} Y^{p-i_0}) = v(b^\prime Y) &\Longrightarrow v(a\a^{i_0} Y^{p-i_0}  ) = v(b^\prime Y) \\
		&\Longrightarrow  v( a\a^{p-1} Y^{p-i_0-1} ) = v (b^\prime \a^{p-i_0-1}) \Longrightarrow  v( b Y^{p-i_0-1} ) = v (b^\prime \a^{p-i_0-1})   \\
		&\Longrightarrow \g = vY = \dfrac{1}{p-i_0-1} v \mathlarger{\mathlarger{\mathlarger{ ( }}} \dfrac{b^\prime \a^{p-i_0-1}}{b} \mathlarger{\mathlarger{\mathlarger{ ) }}}  .
	\end{align*} 
	Similarly, the condition $v(a^\prime_{i_0} Y^{p-i_0}) \leq vc^\prime$ implies that $\g\leq \dfrac{1}{p-i_0} v \mathlarger{\mathlarger{\mathlarger{ ( }}} \dfrac{c^\prime \a^{p-i_0-1}}{b} \mathlarger{\mathlarger{\mathlarger{ ) }}} $. The case when $v(a^\prime_{i_0} Y^{p-i_0}) = vc^\prime \leq v(b^\prime Y)$ is analysed similarly. 
	
	\pars The fact that $v\a < \g$ implies that $v\a^{p-i_0-1} < (p-i_0-1)\g \leq v \mathlarger{\mathlarger{\mathlarger{ ( }}} \dfrac{b^\prime \a^{p-i_0-1}}{b} \mathlarger{\mathlarger{\mathlarger{ ) }}} $. Consequently, we obtain that $v\dfrac{b^\prime}{b} > 0$, that is, $\dfrac{b^\prime}{b}v = 0v$. Observe that $b^\prime = ap\a^{p-1}+b$. It follows that $\dfrac{ap\a^{p-1}}{b}v = -1v$ and hence $\dfrac{a\a^{p-1}}{b}v = -\dfrac{1}{p}v$. Similarly, the condition $v\a < \g \leq \dfrac{1}{p-i_0} v \mathlarger{\mathlarger{\mathlarger{ ( }}} \dfrac{c^\prime \a^{p-i_0-1}}{b} \mathlarger{\mathlarger{\mathlarger{ ) }}} $ implies that $v(b\a) < vc^\prime = v(a\a^p+b\a+c)$. It follows that $\dfrac{a\a^p+c}{b\a}v = -1v$. Utilising the relation $\dfrac{a\a^{p-1}}{b}v = -\dfrac{1}{p}v$, we conclude that $\dfrac{c}{b\a}v = \dfrac{1-p}{p}v$.  
\end{proof}

\begin{Corollary}\label{Coro non-ruled non-Gauss}
	Assume that $k$ is separable-algebraically closed, $\ch k = \ch kv$ and $\dfrac{1}{p-1} v(\dfrac{b}{a}) \leq \dfrac{1}{p} v(\dfrac{c}{a})$. Then for every $\a\in k$, there exists at most one $\g\in vk$ such that the extension with $(\a,\g)$ as a minimal pair of definition is not a Gau{\ss} valuation and the residue field extension $Kv|kv$ is not ruled. 
\end{Corollary}

We can now give a \emph{proof of Proposition \ref{Prop non-ruled non-Gauss}}:

\begin{proof}
	If $v$ is a Gau{\ss} valuation such that the residue field extension is non-ruled, then $vX = \dfrac{1}{p-1} v(\dfrac{b}{a})$ by Theorem \ref{Thm non ruled necessary} and hence $v$ is uniquely determined. It is thus sufficient to show that there exists at most one extension of $v|_k$ from $k$ to $K$ which is not a Gau{\ss} valuation and the residue field extension is non-ruled. 
	
	\pars If the residue field extension is ruled for every extension of $v|_k$ from $k$ to $K$ which is not a Gau{\ss} valuation, then we are done. Hence we assume that there exists at least one extension admitting a non-ruled residue field extension which is also not a Gau{\ss} valuation. Suppose that there exist two such distinct extensions $w$ and $w_1$. Take the corresponding minimal pairs of definition $(\a,\g)$ and $(\a_1,\g_1)$. Without any loss of generality, we assume that
	\[ \g\leq \g_1.  \]
	It follows from Lemma \ref{Lemma non-ruled non-Gauss} that $\dfrac{c}{b\a} v = \dfrac{1-p}{p}v  = \dfrac{c}{b\a_1}v$. The assumption that $\ch k$ does not divide $p-1$ implies that $\dfrac{1-p}{p}v\neq 0$. Thus $\dfrac{b\a}{c}v = \dfrac{b\a_1}{c}v$. Consequently, $\dfrac{\a}{\a_1}v = 1v$ and hence 
	\begin{equation}\label{eqn 24}
		v(\a - \a_1) >  v\a = v\a_1.
	\end{equation} 
	Set $ \a^\prime:= \a_1 - \a $, $Y:= X-\a$ and $Y_1:= X-\a_1$. Write 
	\[  aX^p+bX+c = \sum_{i=0}^{p-2} a^\prime_i Y^{p-i} + b^\prime Y + c^\prime  =  \sum_{i=0}^{p-2} a^\prime_{i,1} Y_1^{p-i} + b_1^\prime Y_1 + c_1^\prime.  \] 
	Suppose that $v\a^\prime \geq \g$. Then $(\a_1,\g)$ is also a minimal pair of definition for $w$ by [\ref{AP sur une classe}, Proposition 3]. As a consequence, we obtain from Corollary \ref{Coro non-ruled non-Gauss} that $\g = \g_1$, that is, $w = w_1$. However, this contradicts the fact that $w$ and $w_1$ are distinct. Thus $v\a^\prime < \g$. We then obtain from Lemma \ref{Lemma non-ruled non-Gauss} that
	\begin{equation*}
		v\a^\prime < \g \leq   \dfrac{1}{p-i_0-1} v \mathlarger{\mathlarger{\mathlarger{ ( }}}  \dfrac{b^\prime \a^{p-i_0-1}}{b} \mathlarger{\mathlarger{\mathlarger{ ) }}}. 
	\end{equation*}
	As a consequence, 
	\begin{equation}\label{eqn 25}
		v (a\a^{i_0}(\a^\prime)^{p-i_0-1}) < v \mathlarger{\mathlarger{\mathlarger{ ( }}} \dfrac{a\a^{p-1} b^\prime }{b} \mathlarger{\mathlarger{\mathlarger{ ) }}} = vb^\prime,
	\end{equation}
	where the equality follows from the fact that $vb = v(aX^{p-1}) = v(a\a^{p-1})$. Observe that
	\begin{align*}
		b_1^\prime &= ap\a_1^{p-1}+b = ap(\a+\a^\prime)^{p-1}+b\\
		&= (ap\a^{p-1} + b) + \sum_{i=0}^{p-2} \binom{p-1}{i} ap\a^i(\a^\prime)^{p-1-i}\\
		&= b^\prime + \sum_{i=0}^{p-2}  \binom{p-1}{i} ap\a^i(\a^\prime)^{p-1-i}. 
	\end{align*}
	Assume that $\binom{p-1}{i}, \binom{p-1}{j} \neq 0$. It follows from (\ref{eqn 24}) that
	\[   v  \mathlarger{\mathlarger{\mathlarger{ ( }}} \binom{p-1}{i} ap\a^i(\a^\prime)^{p-1-i} \mathlarger{\mathlarger{\mathlarger{ ) }}} = v(a\a^i(\a^\prime)^{p-1-i}) > v(a\a^j(\a^\prime)^{p-1-j}) = v \mathlarger{\mathlarger{\mathlarger{ ( }}}  \binom{p-1}{j} ap\a^j(\a^\prime)^{p-1-j} \mathlarger{\mathlarger{\mathlarger{ ) }}}   \]
	whenever $i<j$. Employing the triangle inequality, we then obtain that
	\begin{equation}\label{eqn 26}
		v  \mathlarger{\mathlarger{\mathlarger{ ( }}}  \sum_{i=0}^{p-2}  \binom{p-1}{i} ap\a^i(\a^\prime)^{p-1-i} \mathlarger{\mathlarger{\mathlarger{ ) }}}  = v \mathlarger{\mathlarger{\mathlarger{ ( }}} \binom{p-1}{j} ap\a^j(\a^\prime)^{p-1-j} \mathlarger{\mathlarger{\mathlarger{ ) }}},
	\end{equation}
	where 
	\[  j:= \max \mathlarger{\mathlarger{\mathlarger{\{}}}  i  \mathlarger{\mathlarger{\mathlarger{\mid}}}  0\leq i\leq p-2 \text{ and } \binom{p-1}{i} \neq 0  \mathlarger{\mathlarger{\mathlarger{\}}}}. \]
	The binomial coefficients satisfy the relation $p \binom{p-1}{j} = (p-j) \binom{p}{j}$. The facts that $\binom{p-1}{j}\neq 0$ and $p\neq \ch k$ then imply that $\binom{p}{j}\neq 0$. As a consequence, $j\leq i_0$ where $i_0$ is defined as in Lemma \ref{Lemma non-ruled non-Gauss}. If $j <  i_0$, then $\binom{p-1}{i_0} = 0$ by definition of $j$. The relation $p \binom{p-1}{i_0} = (p-i_0) \binom{p}{i_0}$ then implies that $p-i_0=0$, that is, $p=i_0$. But this contradicts the fact that $i_0\leq p-2$. Thus $j=i_0$. It then follows from (\ref{eqn 25}) that
	\begin{equation}\label{eqn 27}
		v \mathlarger{\mathlarger{\mathlarger{ ( }}} \binom{p-1}{j} ap\a^j(\a^\prime)^{p-1-j} \mathlarger{\mathlarger{\mathlarger{ ) }}} =  v (a\a^{i_0}(\a^\prime)^{p-i_0-1}) < vb^\prime.
	\end{equation}
	Utilising (\ref{eqn 26}), (\ref{eqn 27}), the expansion of $b_1^\prime$, and employing the triangle inequality, we conclude that  
	\[  vb_1^\prime  = v (a\a^{i_0}(\a^\prime)^{p-i_0-1}) < vb^\prime. \]
	Similar arguments also yield that
	\[ vc_1^\prime < vc^\prime.  \]
	As a consequence, we obtain from Lemma \ref{Lemma non-ruled non-Gauss} that 
	\begin{align*}
		\g_1 =  &\min \mathlarger{\mathlarger{\mathlarger{\{}}} \dfrac{1}{p-i_0-1} v \mathlarger{\mathlarger{\mathlarger{ ( }}}  \dfrac{b_1^\prime \a_1^{p-i_0-1}}{b} \mathlarger{\mathlarger{\mathlarger{ ) }}} , \dfrac{1}{p-i_0} v \mathlarger{\mathlarger{\mathlarger{ ( }}}  \dfrac{c_1^\prime \a_1^{p-i_0-1}}{b} \mathlarger{\mathlarger{\mathlarger{ ) }}}   \mathlarger{\mathlarger{\mathlarger{\}}}}   \\
		\mathlarger{\mathlarger{  <  }} & \min \mathlarger{\mathlarger{\mathlarger{\{}}} \dfrac{1}{p-i_0-1} v \mathlarger{\mathlarger{\mathlarger{ ( }}}  \dfrac{b^\prime \a^{p-i_0-1}}{b} \mathlarger{\mathlarger{\mathlarger{ ) }}} , \dfrac{1}{p-i_0} v \mathlarger{\mathlarger{\mathlarger{ ( }}}  \dfrac{c^\prime \a^{p-i_0-1}}{b} \mathlarger{\mathlarger{\mathlarger{ ) }}}   \mathlarger{\mathlarger{\mathlarger{\}}}}   =\g,
	\end{align*}
	which contradicts the fact that $\g\leq\g_1$. The proposition now follows.
\end{proof}


\section{Examples}\label{Sect egs}

Theorem \ref{Thm non ruled necessary} gives necessary conditions for the residue field extension $Kv|kv$ to be non-ruled. However, the conditions are not sufficient, as illustrated by the next few examples.

\begin{Example}\label{Eg ruled Gauss}
	Let $k = \QQ(\zeta_p)$ where $p$ is an odd prime and $\zeta_p$ is a primitive $p$-th root of unity. Let $K= k(X)(\sqrt[p]{qX^p+1})$ where $q$ is a prime such that $q\neq p$. Thus $a=q, b=0$ and $c=1$. Take the $q$-adic valuation $v:= v_q$ on $\QQ$. Then $v\QQ = \ZZ$. Observe that
	\[ \dfrac{1}{p-1} v(\dfrac{b}{a}) = \infty > -\dfrac{1}{p} = \dfrac{1}{p} v(\dfrac{1}{q}) = \dfrac{1}{p} v(\dfrac{c}{a}). \]
	Take an extension of $v$ to $k$. Take the Gau{\ss} valuation extending $v$ to $k(X)$ with $\g := vX = -\dfrac{1}{p}$, and extend $v$ to $K$. Observe that $\g$ has order $p$ modulo $v\QQ$. Suppose that $\g\in vk$. Then $(vk:v\QQ) \geq p$. As a consequence of the Fundamental Inequality, we then obtain that $p\leq (vk:v\QQ) \leq [k:\QQ] = p-1$, which yields a contradiction. It follows that $\g\notin vk$. Thus $vk+\ZZ\g \neq vk$. The residue field extension $Kv|kv$ is then a ruled extension by Proposition \ref{Prop characterise non-ruled}. Observe that in this example condition $(i)$ of Proposition \ref{Prop characterise non-ruled} is not satisfied while condition $(ii)$ is satisfied.
\end{Example}

\begin{Example}
	Let $(k,v)$ be as in Example \ref{Eg ruled Gauss} and $K= k(X)(\sqrt[p]{X^p+q})$. Then $a=1, b=0$ and $c=q$. Observe that 
	\[ \dfrac{1}{p-1} v(\dfrac{b}{a}) = \infty > \dfrac{1}{p} = \dfrac{1}{p} v(\dfrac{c}{a}) = \dfrac{1}{p} vc.  \]
	Take the Gau{\ss} valuation extending $v$ to $k(X)$ with $vX = \dfrac{1}{p}$ and extend $v$ to $K$. We have observed in Example \ref{Eg ruled Gauss} that $\dfrac{1}{p} \notin vk$. Thus condition $(ii)$ of Proposition \ref{Prop characterise non-ruled} is not satisfied, and hence the residue field extension $Kv|kv$ is ruled. Observe that condition $(i)$ of Proposition \ref{Prop characterise non-ruled} is satisfied in this case.
\end{Example}

\begin{Remark}
	It is not possible to construct an example of ruled residue field extensions where both conditions $(i)$ and $(ii)$ of Proposition \ref{Prop characterise non-ruled} are not satisfied. Indeed, the failure of condition $(ii)$ implies that $\dfrac{1}{p} v(\dfrac{c}{a}) \notin vk$ but $\dfrac{1}{p} va \in vk$. This necessarily implies that $vk + \ZZ \dfrac{1}{p} v(\dfrac{c}{a}) = vk+ \ZZ(\dfrac{1}{p} vc  - \dfrac{1}{p} va) = vk + \ZZ\dfrac{1}{p}vc$, that is, condition $(i)$ is satisfied.
\end{Remark}

\begin{Example}
	Let $(k,v)$ be as in Example \ref{Eg ruled Gauss} and $K= k(X)(\sqrt[p]{X^p+1})$. Then $a=c=1$ and $b=0$. Observe that 
	\[ \dfrac{1}{p-1} v(\dfrac{b}{a}) = \infty > 0 = \dfrac{1}{p} v(\dfrac{c}{a}).  \]
	Take the Gau{\ss} valuation extending $v$ to $k(X)$ with $vX = 0$ and extend $v$ to $K$. Then both conditions $(i)$ and $(ii)$ of Proposition \ref{Prop characterise non-ruled} are satisfied and hence the residue field extension $Kv|kv$ is non-ruled. 
\end{Example}

\begin{Example}\label{Eg ruled not Gauss}
	Take the valued field $k:= \QQ(\zeta_p)(T)$ equipped with the $T$-adic valuation, where $\zeta_p$ is a primitive $p$-th root of unity and $T$ is transcendental over $\QQ(\zeta_p)$. Then $vk=\ZZ$. Set
	\[ K:= k(X) (\sqrt[p]{T^pX^p - pTX + T^p+p-1}). \]
	Thus $a = T^p$, $b = -pT$ and $c = T^p+p-1$. Observe that $\dfrac{1}{p-1} v(\dfrac{b}{a}) = -1 = \dfrac{1}{p} v(\dfrac{c}{a})$. Take the valuation extending $v$ to $k(X)$ with the minimal pair of definition $(\dfrac{1}{T}, \dfrac{p}{2}-1)$ and extend it to $K$. Then $v(X-\dfrac{1}{T}) = \dfrac{p}{2}-1 > -1$ and hence $vX = v\dfrac{1}{T} =-1$ by the triangle inequality. Thus 
	\[ vX = \dfrac{1}{p-1} v(\dfrac{b}{a}) = \dfrac{1}{p} v(\dfrac{c}{a}). \]
	Setting $Y:= X - \dfrac{1}{T}$, we then observe that $vY$ has order $2$ modulo $vk$. It now follows from [\ref{APZ characterization of residual trans extns}, Corollary 2.3] that $k(X)v = kv\, (Z)$ where $Z:= \dfrac{Y^2}{T^{p-2}}v$. Further, observe that we can write 
	\[ K = k(X) \mathlarger{\mathlarger{\mathlarger{(}}}  \sqrt[p]{\dfrac{T^pX^p - pTX + T^p+p-1}{T^{p}}}  \mathlarger{\mathlarger{\mathlarger{)}}}. \]
	Expanding the numerator within the radical sign in terms of $Y$, we obtain
	\begin{equation}\label{eqn 28}
		T^pX^p - pTX + T^p+p-1 = \sum_{i=0}^{p-2} \binom{p}{i} T^{p-i}Y^{p-i} + T^p.
	\end{equation}
	Observe that $v \binom{p}{i} T^{p-i}Y^{p-i} = \dfrac{p(p-i)}{2} \geq p = vT^p$ where equality holds only when $i=p-2$. Thus $v \dfrac{T^pX^p - pTX + T^p+p-1}{T^{p}} = 0$. Moreover,  $\sqrt[p]{\dfrac{T^pX^p - pTX + T^p+p-1}{T^{p}}v} \in Kv$. It follows from the expression in (\ref{eqn 28}) that
	\[ \dfrac{T^pX^p - pTX + T^p+p-1}{T^{p}}v = \binom{p}{p-2} \dfrac{Y^2}{T^{p-2}}v + 1 = nZ +1, \]
	where $n:= \binom{p}{p-2} \in \ZZ_{\geq 1}$. As a consequence, we have the chain
	\[ kv\, (Z) \subsetneq kv \, (Z, \sqrt[p]{nZ+1}) \subseteq Kv.  \]
	Observe that $kv\, (Z, \sqrt[p]{nZ+1}) = kv\, (\sqrt[p]{nZ+1})$. Since $[Kv:k(X)v] \leq [K:k(X)]=p = [kv \, (\sqrt[p]{nZ+1}):kv\, (Z)]$, we conclude that $Kv = kv\, (\sqrt[p]{nZ+1})$ and hence $Kv|kv$ is a ruled extension.
\end{Example}

Modifying Example \ref{Eg ruled not Gauss} slightly, we construct examples of non-ruled residue field extensions such that $v$ is not a Gau{\ss} valuation. 

\begin{Example}\label{Eg not ruled not Gauss}
	Let $(k,v)$ be as in Example \ref{Eg ruled not Gauss}. Further, assume that $p$ is an odd prime. Set $K:= k(X)(\sqrt[p]{(T^p-T^{2p})X^p - pTX + T^p+p-1})$. Thus $a = T^p-T^{2p}$, $b = -pT$ and $c = T^p+p-1$. Observe that $\dfrac{1}{p-1} v(\dfrac{b}{a}) = -1 = \dfrac{1}{p} v(\dfrac{c}{a})$. Take the extension of $v$ to $k(X)$ with the minimal pair $(\dfrac{1}{T}, p-1)$ and extend it to $K$. The fact that $v\dfrac{1}{T} = -1 < p-1$ implies that this is not a Gau{\ss} valuation. Moreover, it follows from the triangle inequality that $vX = v\dfrac{1}{T} = -1$. Thus
	\[ vX = \dfrac{1}{p-1} v(\dfrac{b}{a}) = \dfrac{1}{p} v(\dfrac{c}{a}).  \]
	Set $Y:= X-\dfrac{1}{T}$. It follows from [\ref{APZ characterization of residual trans extns}, Corollary 2.3] that $k(X)v = kv\, (Z)$ where $Z:= \dfrac{Y}{T^{p-1}}v$. 
	Observe that
	\[ K = k(X) \mathlarger{\mathlarger{\mathlarger{(}}}  \sqrt[p]{\dfrac{(T^p-T^{2p})X^p - pTX + T^p+p-1}{T^{2p}}}  \mathlarger{\mathlarger{\mathlarger{)}}}. \]
	Expanding $(T^p-T^{2p})X^p - pTX + T^p+p-1$ in terms of $Y$, we obtain that
	\begin{equation}\label{eqn 29}
		(T^p-T^{2p})X^p - pTX + T^p+p-1 =  \sum_{i=0}^{p-2} \binom{p}{i} (T^{p-i} - T^{2p-i})Y^{p-i} - pT^{p+1}Y.
	\end{equation}
	Observe that $v \binom{p}{i} (T^{p-i} - T^{2p-i})Y^{p-i} = p(p-i) \geq 2p = v (pT^{p+1}Y) = vT^{2p}$ for all $i\leq p-2$, where the equality holds only for $i = p-2$. It follows that $v\dfrac{(T^p-T^{2p})X^p - pTX + T^p+p-1}{T^{2p}} = 0$. Moreover, we obtain from the expression in (\ref{eqn 29}) that
	\[ \dfrac{(T^p-T^{2p})X^p - pTX + T^p+p-1}{T^{2p}}v = \binom{p}{p-2} \dfrac{Y^2}{T^{2p-2}}v - p\dfrac{Y}{T^{p-1}}v = nZ^2 -pZ, \]
	where $n:= \binom{p}{p-2}\in\ZZ_{\geq 1}$. We have thus obtained a chain
	\[ k(X)v = kv\, (Z) \subsetneq kv \, (Z, \sqrt[p]{nZ^2 -pZ}) \subseteq Kv. \]
	Since $[Kv:k(X)v] \leq [K:k(X)]=p=[kv \, (Z, \sqrt[p]{nZ^2 -pZ}):kv\, (Z)]$, we conclude that $Kv = kv \, (Z, \sqrt[p]{nZ^2-pZ})$. Applying Lemma \ref{Lemma non ruled char 0} we then obtain that $Kv|kv$ is not ruled.
\end{Example}




\section*{Acknowledgements}
 
Preliminary versions of this paper were written when the author was a Post-Doctoral Fellow (supported by the Post-Doctoral Fellowship of the National Board of Higher Mathematics, India) in IISER Mohali. The author is presently supported by the Seed grant F.28-3(15)/2021-22/SP114 from IIT Bhubaneswar. The author would like to thank the anonymous referee for their many insightful suggestions and comments which has improved the readability and exposition of the paper. 



\begin{thebibliography}{1000000000}
	\bibitem{AP88} V. Alexandru and A. Zaharescu, Sur une classe de prolongements \`{a} $K(x)$ d'une valuation sur une corp $K$, Rev. Roumaine Math. Pures Appl., 5 (1988), 393--400. \label{AP sur une classe}  	
	\bibitem{APZ1} V. Alexandru, N. Popescu and A. Zaharescu, A theorem of characterization of residual transcendental extensions of a valuation, J. Math. Kyoto University, 28 (1988), 579--592. \label{APZ characterization of residual trans extns}
    \bibitem{PGKJB} K. J. Becher and P. Gupta, A ruled residue theorem for function fields of conics, Journal of Pure and Applied Algebra, 225(6) (2021), DOI: \url{https://doi.org/10.1016/j.jpaa.2020.106638}. \label{Gupta Becher ruled residue conics} 
	\bibitem{Du21} A. Dutta, Minimal pairs, minimal fields and implicit constant fields, Journal of Algebra, 588 (2021), 479--514. \label{Dutta min fields implicit const fields}
	\bibitem{DuKu} A. Dutta and F.-V. Kuhlmann, Eliminating tame ramification: generalizations of Abhyankar's lemma, Pacific Journal of Mathematics, 307(1) (2020), 121--136. \label{Dutta Kuh abhyankars lemma}
	\bibitem{EnPr05} A.J. Engler and A. Prestel, Valued Fields, Springer Monographs in Mathematics, Springer-Verlag, Berlin, 2005. \label{Engler book}
	\bibitem{K1} F.-V. Kuhlmann,   Valuation theoretic and model theoretic aspects of local uniformization, in Resolution of Singularities -
		A Research Textbook in Tribute to Oscar Zariski, H. Hauser, J. Lipman, F. Oort, A. Quiros (es.), Progress in Math. 181, Birkh\"auser (2000), 4559-4600. \label{Kuh vln model}
    \bibitem{Lang} S. Lang, Algebra, revised third ed., Addison-Wesley Publishing Company Advanced Book Program, Reading, MA, (2002). \label{Lang}
    \bibitem{Ohm-83} J. Ohm, The ruled residue theorem for simple transcendental extensions of valued fields, Proc. Amer. Math. Soc., 89 (1983), 16--18. \label{Ohm}
\bibitem{HS} H. Stichtenoth, Algebraic Function Fields and Codes, Graduate Texts in Mathematics 254, Springer-Verlag Berlin Heidelberg 2009. \label{Stichtenoth book}
	\bibitem{ZS2} O. Zariski and P. Samuel, Commutative Algebra, Volume II, Van Nostrand, (1958). \label{ZS2}
\end{thebibliography}
\end{document}